\documentclass[11pt]{article}
\usepackage{hyperref}
\usepackage{times}  
\usepackage{mathpazo}
\usepackage{amssymb,amsmath,amsthm}
\usepackage{epsfig}
\usepackage{enumerate}

\newtheorem{theorem}{Theorem}[section]
\newtheorem{proposition}[theorem]{Proposition}
\newtheorem{definition}[theorem]{Definition}

\newtheorem{lemma}[theorem]{Lemma}

\newtheorem{corollary}[theorem]{Corollary}

\newtheorem{remark}[theorem]{Remark}

\newcommand{\qedsymb}{\hfill{\rule{2mm}{2mm}}}
\renewenvironment{proof}[1][]{\begin{trivlist}
\item[\hspace{\labelsep}{\bf\noindent Proof#1:\/}] }{\qedsymb\end{trivlist}}

\def\calF{{\cal F}}

\def\calH{{\cal H}}
\def\Z{{\mathbb{Z}}}
\def\R{\mathbb{R}}

\def\C{\mathbb{C}}

\newcommand\Prob[2]{{\Pr_{#1}\left[ {#2} \right]}}

\newcommand{\SAT}{\mathsf{SAT}}

\newcommand{\NP}{\mathsf{NP}}

\renewcommand{\P}{\mathsf{P}}

\newcommand{\od}{\overline{\xi}}

\renewcommand{\epsilon}{\varepsilon}

\newcommand{\rank}{\mathop{\mathrm{rank}}}
\newcommand{\minrank}{\mathop{\mathrm{minrk}}}

\newcommand{\linspan}{\mathop{\mathrm{span}}}
\newcommand{\Fset}{\mathbb{F}}         

\begin{document}

\title{{\bf Local Orthogonality Dimension}}

\author{
Inon Attias\thanks{School of Computer Science, The Academic College of Tel Aviv-Yaffo, Tel Aviv 61083, Israel.
} \qquad
Ishay Haviv\thanks{School of Computer Science, The Academic College of Tel Aviv-Yaffo, Tel Aviv 61083, Israel. Research supported in part by the Israel Science Foundation (grant No.~1218/20).
}
}

\date{}

\maketitle

\begin{abstract}
An orthogonal representation of a graph $G$ over a field $\Fset$ is an assignment of a vector $u_v \in \Fset^t$ to every vertex $v$ of $G$, such that $\langle u_v,u_v \rangle \neq 0$ for every vertex $v$ and $\langle u_v,u_{v'} \rangle = 0$ whenever $v$ and $v'$ are adjacent in $G$. The locality of the orthogonal representation is the largest dimension of a subspace spanned by the vectors associated with a closed neighborhood in the graph.
We introduce a novel graph parameter, called the {\em local orthogonality dimension}, defined for a given graph $G$ and a given field $\Fset$, as the smallest possible locality of an orthogonal representation of $G$ over $\Fset$. This is a local variant of the well-studied orthogonality dimension of graphs, introduced by Lov{\'a}sz (Trans. Inf. Theory,~1978), analogously to the local variant of the chromatic number introduced by Erd\"{o}s et al.~(Discret. Math.,~1986).

We investigate the usefulness of topological methods for proving lower bounds on the local orthogonality dimension.
Such methods are known to imply tight lower bounds on the chromatic number of several graph families of interest, such as Kneser graphs and Schrijver graphs.
We prove that graphs for which topological methods imply a lower bound of $t$ on their chromatic number have local orthogonality dimension at least $\lceil t/2 \rceil +1$ over every field, strengthening a result of Simonyi and Tardos on the local chromatic number (Combinatorica,~2006). We show that for certain graphs this lower bound is tight, whereas for others, the local orthogonality dimension over the reals is equal to the chromatic number. More generally, we prove that for every complement of a line graph, the local orthogonality dimension over $\R$ coincides with the chromatic number.
This strengthens a recent result by Daneshpajouh, Meunier, and Mizrahi, who proved that the local and standard chromatic numbers of these graphs are equal (J.~Graph Theory,~2021).
As another extension of their result, we prove that the local and standard chromatic numbers are equal for some additional graphs, from the family of Kneser graphs.
We also study the computational aspects of the local orthogonality dimension and show that for every integer $k \geq 3$ and a field $\Fset$, it is $\NP$-hard to decide whether the local orthogonality dimension of an input graph over $\Fset$ is at most $k$.
We finally present an application of the local orthogonality dimension to the index coding problem from information theory, extending a result of Shanmugam, Dimakis, and Langberg (ISIT,~2013).

\end{abstract}

\newpage
\section{Introduction}

In this work, we introduce a novel graph parameter called the {\em local orthogonality dimension}.
We investigate it from combinatorial and computational perspectives and relate it to the index coding problem from information theory.
We start with an overview on several related graph parameters and then present our contribution.

\paragraph*{Chromatic number \& local chromatic number.}
Graph coloring is one of the most fundamental and popular topics in graph theory.
A proper coloring of a (simple undirected) graph $G$ is an assignment of a color to each vertex such that every two adjacent vertices receive distinct colors. The {\em chromatic number} of $G$, denoted by $\chi(G)$, is the minimum number of colors needed for a proper coloring of $G$.
A variant of the chromatic number, known as the {\em local chromatic number}, was introduced in 1986 by Erd\"{o}s, F\"{u}redi, Hajnal, Komj\'{a}th, R\"{o}dl, and Seress~\cite{ErdosLocal}.
In contrast to the standard chromatic number, here the objective is not to minimize the total number of colors of a proper coloring, but the number of colors locally viewed by the vertices. More precisely, the {\em locality} of a proper coloring of a graph $G$ is the maximum number of colors that appear in a closed neighborhood of a vertex of $G$ (where a closed neighborhood of a vertex consists of the vertex itself and its neighbors). The local chromatic number of $G$, denoted by $\chi_l(G)$, is defined as the smallest possible locality of a proper coloring of $G$.
Every proper coloring of a graph $G$ with $\chi(G)$ colors obviously has locality at most $\chi(G)$, hence $\chi_l(G) \leq \chi(G)$ for every graph $G$.
It was shown in~\cite{ErdosLocal} that the gap between the two quantities can be arbitrarily large, even for graphs with local chromatic number $3$.
From a computational point of view, for every $k \geq 3$, it is $\NP$-hard to decide whether an input graph $G$ satisfies $\chi_l(G) \leq k$~\cite{Osang}, and strong hardness of approximation results are known to follow from those of the standard chromatic number.

In 1978, Lov{\'a}sz~\cite{LovaszKneser} developed an exciting approach for proving lower bounds on the chromatic number of graphs relying on the Borsuk-Ulam theorem from algebraic topology~\cite{Borsuk33}. The approach was applied in~\cite{LovaszKneser} to prove a conjecture of Kneser~\cite{Kneser55}, saying that for all integers $n \geq 2k$, it holds that $\chi(K(n,k)) = n-2k+2$, where $K(n,k)$ is the {\em Kneser graph} defined as follows. Its vertices are all the $k$-subsets of $[n] = \{1,2,\ldots,n\}$, and two such sets are adjacent in the graph if they are disjoint.
The result of~\cite{LovaszKneser} was strengthened by Schrijver~\cite{Schrijver78}, who proved that the subgraph of $K(n,k)$ induced by the $k$-subsets of $[n]$ that do not include two consecutive integers modulo $n$ has the same chromatic number. This subgraph is known as the {\em Schrijver graph} and is denoted by $S(n,k)$.
The topological approach of~\cite{LovaszKneser} was further generalized to obtain lower bounds on the chromatic number of general graphs (see, e.g.,~\cite{MatousekZ04,MatousekBook07}).
We say that a graph $G$ is {\em topologically $t$-chromatic} if a certain topological argument, specified later, implies that $\chi(G) \geq t$.
This terminology, borrowed from Simonyi and Tardos~\cite{SimonyiT06}, allows us to describe statements in their full generality, and yet, its precise definition is not essential throughout this introduction, so it is deferred to Section~\ref{sec:top_bound} (see Remark~\ref{remark:topologically}). The reader is encouraged to think of topologically $t$-chromatic graphs as the graphs $K(n,k)$ and $S(n,k)$ with $t=n-2k+2$.
Other known families of topologically $t$-chromatic graphs $G$ satisfying $\chi(G)=t$ are Borsuk graphs, generalized Mycielski graphs, and rational complete graphs. For additional examples, see, e.g.,~\cite[Section~3.3]{SimonyiT07}.

In a line of works initiated by Simonyi and Tardos~\cite{SimonyiT06}, topological methods were also employed to prove lower bounds on the local chromatic number.
A result of Simonyi, Tardif, and Zsb{\'{a}}n~\cite{SimonyiTZ13}, extending a previous result of~\cite{SimonyiT06}, asserts the following (see Remark~\ref{remark:topologically}).
\begin{theorem}[\cite{SimonyiT06,SimonyiTZ13}]\label{thm:SimonyiT_t/2}
For every topologically $t$-chromatic graph $G$ with at least one edge,
\[\chi_l(G) \geq \lceil t/2 \rceil +1.\]
\end{theorem}
\noindent
For topologically $t$-chromatic graphs $G$ with $\chi(G)=t$, such as the graphs $K(n,k)$ and $S(n,k)$ with $t=n-2k+2$, Theorem~\ref{thm:SimonyiT_t/2} implies that $\chi_l(G)$ lies in the interval $[\lceil t/2 \rceil +1,t]$.
Interestingly, it turns out that both the lower and upper bounds can be tight.
For example, it was shown in~\cite[Section~4.2]{SimonyiT06} that for Schrijver graphs $S(n,k)$ with an odd $t=n-2k+2 > 2$ satisfying $n \geq 4t^2 -7t$, the local chromatic number is $\lceil t/2 \rceil +1$ (see~\cite[Remark~4]{SimonyiT06} and~\cite{SimonyiTV2005} for such statements with relaxed conditions).
Another example  for graphs that attain the lower bound, given in~\cite[Section~4.3]{SimonyiT06}, is from the family of generalized Mycielski graphs.
On the other hand, for the Schrijver graphs $S(n,2)$, it was shown in~\cite[Section~4.2]{SimonyiT06} that the local chromatic number is equal to the chromatic number, that is, $\chi_l(S(n,2)) = \chi(S(n,2)) = n-2$ for all $n \geq 4$, hence the upper bound of $t$ is tight in this case.
This result was recently generalized by Daneshpajouh, Meunier, and Mizrahi~\cite{DMMLine21}, as follows.
\begin{theorem}[\cite{DMMLine21}]\label{thm:DMMLine21}
If $G$ is the complement of a line graph, then $\chi_l(G) = \chi(G)$.
\end{theorem}
\noindent
Recall that the {\em line graph} of a graph $H$ has a vertex for every edge of $H$ and two vertices are adjacent if as edges of $H$ they share a vertex. Notice that the Schrijver graph $S(n,2)$ is the complement of the line graph of the complement of a cycle on $n$ vertices.

\paragraph*{Orthogonality dimension.}
A {\em $t$-dimensional orthogonal representation} of a graph $G=(V,E)$ over a field $\Fset$ is an assignment of a vector $u_v \in \mathbb{F}^t$ with $\langle u_v,u_v \rangle \neq 0$ to every vertex $v \in V$, such that $\langle u_v, u_{v'} \rangle = 0$ whenever $v$ and $v'$ are adjacent vertices in $G$.
Here, for two vectors $x,y \in \Fset^t$, we let $\langle x,y \rangle = \sum_{i=1}^{t}{x_i y_i}$ denote the standard inner product of $x$ and $y$ over $\Fset$, where over the complex field $\C$, it can be replaced by $\langle x,y \rangle = \sum_{i=1}^{t}{x_i \overline{y_i}}$.
The {\em orthogonality dimension} of a graph $G$ over a field $\Fset$, denoted by $\od(G,\Fset)$, is the smallest integer $t$ for which there exists a $t$-dimensional orthogonal representation of $G$ over $\Fset$.\footnote{Orthogonal representations of graphs are sometimes defined in the literature as orthogonal representations of the complement, namely, the definition requires vectors associated with {\em non-adjacent} vertices to have zero inner product. Here we use the other definition, but one may view the notation $\od(G,\Fset)$ as standing for $\xi(\overline{G}, \Fset)$.}
The research on orthogonal representations and on the orthogonality dimension was initiated by Lov{\'a}sz~\cite{Lovasz79} in the study of the Shannon capacity of graphs in information theory. Over the years, they have been found useful for applications in several areas of theoretical computer science, e.g., algorithms~\cite{AlonK98}, circuit complexity~\cite{Valiant77}, and communication complexity~\cite{BrietZ17}.
As for the computational perspective, it follows from a work of Peeters~\cite{Peeters96} that for every $k \geq 3$ and a field $\Fset$, it is $\NP$-hard to decide whether an input graph $G$ satisfies $\od(G,\Fset) \leq k$ (see~\cite{LangbergS08} and~\cite{GolovnevH21} for related hardness of approximation results).
It is worth mentioning here another graph parameter, called {\em minrank} and denoted by ${\minrank}_\Fset (G)$, that was introduced by Haemers in~\cite{Haemers81} motivated by questions in information theory (see Definition~\ref{def:minrank}).
This graph parameter is closely related to the orthogonality dimension, and in particular, it satisfies $\minrank_\Fset(G) \leq \od(\overline{G},\Fset)$, where $\overline{G}$ stands for the complement of $G$ (see Section~\ref{sec:minrank}).

Orthogonal representations can be viewed as a generalization of proper colorings of graphs. Indeed, a proper coloring of a graph with $t$ colors naturally induces a $t$-dimensional orthogonal representation over every field $\Fset$, assigning to the vertices of the $i$th color class the $i$th vector $e_i$ of the standard basis of $\Fset^t$. This implies that for every graph $G$ and a field $\Fset$, $\od(G,\Fset) \leq \chi(G)$.
For the opposite direction, a $t$-dimensional orthogonal representation of a graph $G$ over a finite field $\Fset$ of size $q$ forms a proper coloring of $G$ with at most $q^t$ colors, hence
\begin{eqnarray}\label{eq:chrom_od}
\od(G,\Fset) \geq \log_{q}{\chi(G)}.
\end{eqnarray}
Over the reals, it is not difficult to show that $\od(G,\R) \geq \log_3{\chi(G)}$, and it turns out that there are graphs for which this lower bound on the orthogonality dimension is tight up to a multiplicative constant (see, e.g.,~\cite[Proposition~2.2]{HavivMFCS19}).
Recently, topological methods were found beneficial for proving lower bounds on the orthogonality dimension and on the minrank parameter of graphs.
This approach was initiated in~\cite{Haviv19} and successfully extended by Alishahi and Meunier~\cite{AlishahiM21}, who proved that for every topologically $t$-chromatic graph $G$ with at least one edge and for every field $\Fset$, it holds that
\begin{eqnarray}\label{eq:AlishahiM}
\od(G,\Fset) \geq t \mbox{~~~and~~~} {\minrank}_\Fset(\overline{G}) \geq \lceil t/2 \rceil +1.
\end{eqnarray}

\paragraph*{Index coding.}
The {\em index coding} problem, introduced by Birk and Kol~\cite{BirkKol98} and further developed by Bar-Yossef, Birk, Jayram, and Kol~\cite{BBJK06}, is a well-studied problem in zero-error information theory.
In this problem, a sender holds an $n$-symbol message $x \in \Sigma^n$ over an alphabet $\Sigma$ and wishes to broadcast information to $n$ receivers $R_1,\ldots,R_n$ in a way that enables each receiver $R_i$ to retrieve the $i$th symbol $x_i \in \Sigma$.  For this purpose, the receivers are allowed to use some side information they have in advance comprising a subset of the symbols of $x$. The side information map is naturally represented by a directed graph $G$ on the vertex set $[n]$ that includes a directed edge $(i,j)$ if the receiver $R_i$ knows $x_j$. For simplicity, we will consider here symmetric side information maps and will thus refer to $G$ as undirected. For a given side information graph $G$, the goal is to design an encoding function that maps any $n$-symbol message $x \in \Sigma^n$ to a broadcast information in $\Sigma^\ell$, where $\ell$ is as small as possible, so that the receivers are able to retrieve their messages based on this information and on the side information available to them.
As example, consider the case where $\Sigma = \Z_m$ and $G$ is the complete graph on $n$ vertices, meaning that every receiver $R_i$ knows all the symbols $x_j$ with $j \in [n] \setminus \{i\}$. Observe that it suffices in this case to broadcast a single symbol of $\Sigma$ that consists of the sum $\sum_{i=1}^{n}{x_i}$ modulo $m$ to enable each receiver to retrieve its message.

A significant attention was given in the literature for the setting of {\em linear} index coding. Here, the alphabet $\Sigma$ is a field $\mathbb{F}$, and the sender is allowed to apply a linear encoding function over $\Fset$ to map the message $x \in \Fset^n$ to the transmitted broadcast information. It was shown in~\cite{BBJK06} that the minimum length of a linear index code over a field $\Fset$ for a side information graph $G$ is precisely the aforementioned minrank parameter of $G$ over $\Fset$.
It is easy to see that ${\minrank}_\Fset (G) \leq \chi(\overline{G})$ for every graph $G$ and a field $\Fset$, as follows from the index code that consists of the sum of the symbols of every clique in a minimum clique cover of $G$. The linear index coding problem was related to the local chromatic number by Shanmugam, Dimakis, and Langberg~\cite{shanmugamKDALM2013}, who proved that ${\minrank}_\Fset (G) \leq \chi_l(\overline{G})$ whenever the field $\Fset$ is sufficiently large (see Proposition~\ref{prop:local_minrank_n} for a precise statement). For the binary field $\Fset_2$, which is of special interest for the index coding problem, they proved the following.
\begin{theorem}[\cite{shanmugamKDALM2013}]\label{thm:minrk_SDL}
For every graph $G$ on $n$ vertices, $\minrank_{\Fset_2}(G) \leq \chi_l (\overline{G}) + 2 \cdot \log_2 n$.
\end{theorem}

\subsection{Our Contribution}

In this work, we initiate the study of a novel graph parameter called the {\em local orthogonality dimension}.
This graph parameter is a local variant of the orthogonality dimension of graphs, analogously to the local variant of the chromatic number introduced in~\cite{ErdosLocal}.
In contrast to the standard orthogonality dimension, here we do not aim to minimize the dimension of the whole space of an orthogonal representation of the graph, but the dimensions of the subspaces spanned by the vectors associated with the vertices of closed neighborhoods.
The formal definition is given below.
Here, for a graph $G$ and a vertex $v$, we let $N(v)$ stand for the set of vertices adjacent to $v$ in $G$.

\begin{definition}\label{def:localOD}
The {\em locality} of an orthogonal representation $(u_v)_{v \in V}$ of a graph $G=(V,E)$ over a field $\Fset$ is the maximum dimension of a subspace spanned by the vectors of a closed neighborhood of a vertex, that is, $\max_{v \in V}{\dim(U_v)}$ where $U_v = \linspan (\{u_{v'} \mid v' \in \{v\} \cup N(v)\})$. The {\em local orthogonality dimension} of a graph $G$ over a field $\Fset$, denoted by $\od_l(G,\Fset)$, is the smallest integer $\ell$ for which there exists an orthogonal representation of $G$ over $\Fset$ with locality $\ell$.
\end{definition}

Aside from being a natural graph parameter to investigate, the study of the local orthogonality dimension is motivated by the connections, discovered by Shanmugam et al.~\cite{shanmugamKDALM2013}, between the index coding problem and the local chromatic number.
Indeed, the same authors showed in~\cite{ShanmugamDL14} that the known graph-theoretic combinatorial upper bounds on the index coding problem, even in the local framework, suffer from an inherent limitation. In contrast, it is known that for certain graphs, algebraic tools such as the orthogonality dimension and the minrank parameter provide significantly better upper bounds~\cite{LS07}. It would thus be of interest to explore the potential benefit of locality for such algebraic tools, as suggested by Definition~\ref{def:localOD}.

Let us already mention some basic observations on the local orthogonality dimension.
For every graph $G$ and a field $\Fset$, it clearly holds that
\begin{eqnarray}\label{eq:od_l_vs_od}
\od_l(G,\Fset) \leq \od(G,\Fset),
\end{eqnarray}
because every $t$-dimensional orthogonal representation of $G$ over $\Fset$ has locality at most $t$.
It further holds that
\begin{eqnarray}\label{eq:od_l_vs_chi_l}
\od_l(G,\Fset) \leq \chi_l(G),
\end{eqnarray}
as a proper coloring of $G$ with locality $\ell$ induces an orthogonal representation over any field $\Fset$ with locality $\ell$, simply by assigning the vector $e_i$ of the standard basis of $\Fset^\ell$ to the vertices of the $i$th color class.
Since there exist graphs $G$ with $\chi_l(G) = 3$ and arbitrarily large $\chi(G)$~\cite{ErdosLocal}, inequality~\eqref{eq:chrom_od} above implies that the orthogonality dimension can be arbitrarily large even when its local variant is $3$.

Our first result, proved in Section~\ref{sec:top_bound}, shows that topological methods can be used to obtain lower bounds on the local orthogonality dimension of graphs.
Its proof is a simple application of a result of~\cite{AlishahiM21}.
\begin{theorem}\label{thm:LocalODtopo}
For every topologically $t$-chromatic graph $G$ with at least one edge and for every field $\Fset$,
\[\od_l(G,\Fset) \geq \lceil t/2 \rceil +1.\]
\end{theorem}
\noindent
By~\eqref{eq:od_l_vs_chi_l}, Theorem~\ref{thm:LocalODtopo} strengthens the lower bound of~\cite{SimonyiT06,SimonyiTZ13} on the local chromatic number, as given in Theorem~\ref{thm:SimonyiT_t/2}.
Recall that certain families of topologically $t$-chromatic graphs $G$ are known to satisfy $\chi(G)=t$ and $\chi_l(G) = \lceil t/2 \rceil+1$ (namely, certain Schrijver graphs and generalized Mycielski graphs~\cite{SimonyiT06}). It thus follows from~\eqref{eq:od_l_vs_chi_l} that the lower bound given in Theorem~\ref{thm:LocalODtopo} is tight on these graphs.
Note that, by~\eqref{eq:AlishahiM}, these graphs also satisfy $\od(G,\Fset) = t$ for every field $\Fset$, showing that the inequality~\eqref{eq:od_l_vs_od} can be strict.
Moreover, by considering a disjoint union of such a graph with a graph from~\cite{ErdosLocal} that has local chromatic number $3$ and large orthogonality dimension over $\Fset$, one can obtain graphs $G$ satisfying $\od_l(G,\Fset) < \min (\od(G,\Fset), \chi_l(G))$.

We next consider the local orthogonality dimension of complements of line graphs.
The following theorem, proved in Section~\ref{sec:line}, shows that for this family of graphs, the local orthogonality dimension over the reals coincides with the chromatic number.
\begin{theorem}\label{thm:xi_l_K(F)}
If $G$ is the complement of a line graph, then $\od_l(G,\R) = \chi(G)$.
\end{theorem}
\noindent
Theorem~\ref{thm:xi_l_K(F)} implies that for every integer $n \geq 4$, the Schrijver graph $S(n,2)$, which is topologically $t$-chromatic for $t=n-2$, satisfies $\od_l (S(n,2),\R) = n-2$. This shows that there are topologically $t$-chromatic graphs $G$ with $\chi(G)=t$ for which the lower bound given in Theorem~\ref{thm:LocalODtopo} has a multiplicative gap of roughly $2$ from the truth.
Additionally, since every graph $G$ satisfies $\od_l(G,\R) \leq \chi_l(G) \leq \chi(G)$, the theorem implies that for complements of lines graphs, the chromatic number is equal to its local variant, and thus strengthens the result of~\cite{DMMLine21} given in Theorem~\ref{thm:DMMLine21}.

In the current work, we contribute another family of graphs for which the chromatic number is equal to the local chromatic number.
This is done in the following theorem, proved in Section~\ref{sec:K(n,3)}, that determines the local chromatic number of the Kneser graphs $K(n,3)$.
\begin{theorem}\label{thm:K(n,3)}
For every $n \geq 6$,~ $\chi_l(K(n,3)) = n-4$.
\end{theorem}
\noindent
The proof of Theorem~\ref{thm:K(n,3)} borrows the approach applied in~\cite{SimonyiT06} to prove that for every $n \geq 4$, $\chi_l(S(n,2)) = n-2$.
It was shown there that for every proper coloring of $S(n,2)$, the vertices of every color class are adjacent to `many' vertices of the graph.
This was used to derive that there exists a vertex whose closed neighborhood includes at least $n-2$ colors.
However, a straightforward attempt to apply this strategy to the graph $K(n,3)$ does not succeed.
It turns out that a proper coloring of this graph may include color classes whose vertices are adjacent to relatively `few' vertices, hence one cannot deduce, as in~\cite{SimonyiT06}, that some vertex has at least $n-4$ colors in its closed neighborhood.
To overcome this difficulty, we show that if a proper coloring of $K(n,3)$ involves such color classes, then the total number of colors must be quite `large', and this is used to obtain a vertex with at least $n-4$ colors in its closed neighborhood.
The proof involves a delicate analysis of the number of vertices adjacent to the vertices of color classes of several types, and uses the fact that the chromatic number of $K(n,3)$ is equal to that of its subgraph $S(n,3)$.

We next consider the computational aspects of the local orthogonality dimension.
It should be mentioned that a result of~\cite{DMMLine21} implies that, unless $\P=\NP$, there is no algorithm with polynomial running-time that determines the chromatic number of the complement of a given line graph (see~\cite[Section~4.3]{DMMLine21}). By Theorem~\ref{thm:xi_l_K(F)}, it follows that it is unlikely that there exists an algorithm with polynomial running-time that determines the local orthogonality dimension over $\R$ of such graphs. For general fields, we prove in Section~\ref{sec:hardness} the following hardness result.
\begin{theorem}\label{thm:hardness}
For every integer $k \geq 3$ and a field $\Fset$, the problem of deciding whether an input graph $G$ satisfies $\od_l(G,\Fset) \leq k$ is $\NP$-hard.
\end{theorem}
\noindent
The proof of Theorem~\ref{thm:hardness} combines ideas that were applied by Peeters~\cite{Peeters96} and by Osang~\cite{Osang} to prove hardness results for, respectively, the orthogonality dimension and the local chromatic number.
We note that for $k \in \{1,2\}$ and for every field $\Fset$, a graph $G$ satisfies $\od_l(G,\Fset) \leq k$ if and only if $\chi(G) \leq k$ (see Lemma~\ref{lemma:od<=2}). This implies that for such values of $k$, it is possible to efficiently decide whether a graph $G$ satisfies $\od_l(G,\Fset) \leq k$ (see Proposition~\ref{prop:k=1,2_easy}).

Finally, in Section~\ref{sec:index_coding}, we relate the local orthogonality dimension to the minrank parameter, motivated by the linear index coding problem.
We first present, in a generalized form, the connection given in~\cite{shanmugamKDALM2013} between the local chromatic number and the minrank parameter, and relate it to a derandomization problem studied by Schulman~\cite{Schulman92} (see also~\cite{HavivL19}).
We observe there that topologically $t$-chromatic graphs $G$ with $\chi_l(G) = \lceil t/2 \rceil+1$, such as the ones given in~\cite{SimonyiT06}, satisfy ${\minrank}_{\Fset}(\overline{G}) = \lceil t/2 \rceil+1$ for every sufficiently large field $\Fset$ (see Proposition~\ref{prop:local_minrank_n}). This shows that for certain graphs and fields, the topological lower bound on the minrank parameter, given in~\eqref{eq:AlishahiM} and proved in~\cite{AlishahiM21}, is tight.
Then, for the local orthogonality dimension, we prove the following.
\begin{theorem}\label{thm:binaryOD}
For every graph $G$ on $n$ vertices, ${\minrank}_{\Fset_2}(G) \leq \od_l(\overline{G},\Fset_2)+\lceil \log_2 n \rceil$.
\end{theorem}
\noindent
Theorem~\ref{thm:binaryOD} is proved by a probabilistic argument. By~\eqref{eq:od_l_vs_chi_l}, it strengthens Theorem~\ref{thm:minrk_SDL} proved in~\cite{shanmugamKDALM2013}.
For an extension of Theorem~\ref{thm:binaryOD} to general finite fields, see Theorem~\ref{thm:minrk_localOD_gen}.

\section{A Topological Lower Bound on the Local Orthogonality Dimension}\label{sec:top_bound}

In this section we show, as a simple application of a result of~\cite{AlishahiM21}, that topological methods can be used to obtain lower bounds on the local orthogonality dimension over every field. We start with the following remark on the notion of topologically $t$-chromatic graphs.

\begin{remark}\label{remark:topologically}
In this work, we use a variant of the notion of topologically $t$-chromatic graphs, suggested in~\cite{SimonyiT06}, to describe in a general form the bounds that topological methods provide on graph parameters. The precise definition of this notion is not needed throughout this work, and we refer the reader to Matou{\v{s}}ek's book~\cite{MatousekBook07} for an in-depth introduction to the topic.
For concreteness, we mention that we refer to a graph $G$ as {\em topologically $t$-chromatic} if $t \leq \mathrm{Xind}(\mathrm{Hom}(K_2,G)) +2$, where $\mathrm{Hom}(K_2,G)$ stands for the homomorphism complex of $K_2$ in $G$ and $\mathrm{Xind}$ for the cross-index (see~\cite[Sections~2 and~3]{SimonyiTZ13} for the definitions).
It should be mentioned, though, that this definition is different from the one used in~\cite{SimonyiT06}. There, a graph $G$ was said to be topologically $t$-chromatic if $t \leq \mathrm{coind}(B_0(G)) +1$, where $B_0(G)$ stands for the box complex of $G$ and $\mathrm{coind}$ for the $\Z_2$-coindex. With respect to this definition, it was proved there that a topologically $t$-chromatic graph $G$ with at least one edge satisfies $\chi_l(G) \geq \lceil t/2 \rceil +1$, and in the later work~\cite{SimonyiTZ13}, the same bound was proved under the assumption $t \leq \mathrm{Xind}(\mathrm{Hom}(K_2,G)) +2$. The result of~\cite{SimonyiTZ13} extends the one of~\cite{SimonyiT06}, because it is known that $\mathrm{coind}(B_0(G)) \leq \mathrm{Xind}(\mathrm{Hom}(K_2,G)) +1$ holds for every graph $G$ (see~\cite[Section~3]{SimonyiTZ13}). Further, the results of~\cite{AlishahiM21}, that relate the topological quantities to orthogonality dimension and minrank, were also shown under the weaker assumption $t \leq \mathrm{Xind}(\mathrm{Hom}(K_2,G)) +2$ (making the results stronger).
This allows us to use the above weaker condition as the definition of topologically $t$-chromatic graphs, where all the stated results on such graphs in particular hold under the stronger condition of~\cite{SimonyiT06}.
\end{remark}

We need the following definition of independent representations of graphs.
\begin{definition}\label{def:ind_rep}
A {\em $t$-dimensional independent representation} of a graph $G=(V,E)$ over a field $\Fset$ is an assignment of a vector $u_v \in \mathbb{F}^t$ to every vertex $v \in V$, such that for every $v \in V$, $u_v$ does not belong to the subspace $\linspan (\{u_{v'} \mid v' \in N(v)\})$ spanned by the vectors of its neighbors.
\end{definition}

\begin{remark}\label{remark:ortho_vs_ind}
Note that an orthogonal representation of a graph $G=(V,E)$ over a field $\Fset$ is also an independent representation of $G$ over $\Fset$.
Indeed, for an orthogonal representation $(u_v)_{v \in V}$ of $G$ over $\Fset$, it is impossible that some vector $u_v$ is a linear combination of the vectors $u_{v'}$ with $v' \in N(v)$, because the inner product of such a linear combination with $u_v$ is zero whereas the inner product of $u_v$ with itself is not.
\end{remark}

We will use the following special case of a result proved in~\cite{AlishahiM21}.
\begin{theorem}[\cite{AlishahiM21}]\label{thm:AlishahiM}
Let $G$ be a topologically $t$-chromatic graph with at least one edge, and let $\Fset$ be a field.
Then, for every independent representation of $G$ over $\Fset$, there exists a complete bipartite subgraph $K_{\lfloor t/2 \rfloor, \lceil t/2 \rceil}$ of $G$ such that the vectors assigned to each of its sides are linearly independent over $\Fset$.
\end{theorem}

We are ready to derive Theorem~\ref{thm:LocalODtopo}.

\begin{proof}[ of Theorem~\ref{thm:LocalODtopo}]
Let $G=(V,E)$ be a topologically $t$-chromatic graph with at least one edge, and let $\Fset$ be a field. Put $\ell = \od_l(G,\Fset)$.
By definition, there exists an orthogonal representation $(u_v)_{v \in V}$ of $G$ over $\Fset$ with locality $\ell$, which forms, by Remark~\ref{remark:ortho_vs_ind}, an independent representation of $G$ over $\Fset$. By Theorem~\ref{thm:AlishahiM}, there exists a complete bipartite subgraph $K_{\lfloor t/2 \rfloor, \lceil t/2 \rceil}$ of $G$ such that the vectors assigned to each of its sides are linearly independent over $\Fset$. In particular, the vectors of the vertices of the right side of this bipartite subgraph span a linear subspace of dimension $\lceil t/2 \rceil$. Let $v$ be an arbitrary vertex from the left side of the bipartite subgraph, and observe that its vector $u_v$ cannot be represented as a linear combination of the vectors of the right side (because the inner product of $u_v$ with itself is nonzero, whereas the inner products of $u_v$ with the vectors of the right side are all zeros). It thus follows that the vectors $\{u_{v'} \mid v' \in \{v\} \cup N(v) \}$ assigned by the given orthogonal representation to the closed neighborhood of $v$ span a subspace of dimension at least $\lceil t/2 \rceil + 1$, hence its locality satisfies $\ell \geq \lceil t/2 \rceil+1$, as desired.
\end{proof}

\section{The Local Orthogonality Dimension of Complements of Line Graphs}\label{sec:line}

In this section we prove Theorem~\ref{thm:xi_l_K(F)}, which asserts that the local orthogonality dimension over $\R$ of the complement of a line graph is equal to its chromatic number.

Consider the following definition.
\begin{definition}\label{def:K(F)}
For a set system $\calF$, let $K(\calF)$ denote the graph on the vertex set $\calF$ where two distinct vertices $A,B \in \calF$ are adjacent if $A \cap B = \emptyset$.
\end{definition}
\noindent
Note that the Kneser graph $K(n,k)$ is the graph $K(\calF)$ where $\calF$ is the family of all $k$-subsets of $[n]$, and that the Schrijver graph $S(n,k)$ is the graph $K(\calF)$ where $\calF$ is the family of all $k$-subsets of $[n]$ with no two consecutive elements modulo $n$.
In fact, it is known that every graph is isomorphic to $K(\calF)$ for some set system $\calF$ (see, e.g.,~\cite{MatousekZ04}).
Observe that the complements of line graphs are precisely the graphs $K(\calF)$ where $\calF$ is a system of sets of size $2$ each.

Throughout the proof of Theorem~\ref{thm:xi_l_K(F)}, we say that two real vectors $x,y \in \R^t$ are {\em proportional} if $x = \alpha \cdot y$ for some nonzero $\alpha \in \R$.

\begin{proof}[ of Theorem~\ref{thm:xi_l_K(F)}]
The complement $G$ of a line graph can be viewed, as mentioned above, as a graph $K(\calF)$ for a family $\calF$ of $2$-subsets of some ground set $[n]$ (see Definition~\ref{def:K(F)}).
By~\eqref{eq:od_l_vs_chi_l}, it holds that $\od_l(G,\R) \leq \chi_l(G) \leq \chi(G)$.
To prove that $\od_l(G,\R) \geq \chi(G)$, we will apply an induction on the number $n$ of the elements in the ground set.
For $n=2$, the statement is trivial because $G$ has at most one vertex, hence it clearly holds that $\od_l(G,\R) = \chi(G)$.
We turn to prove the statement for $n \geq 3$, assuming that it holds for $n-1$.

Let $\calF$ be a system of $2$-subsets of $[n]$, and let $G = K(\calF)$.
Put $\ell = \od_l(G,\R)$, and let $(u_A)_{A \in \calF}$ be a $t$-dimensional orthogonal representation $(u_A)_{A \in \calF}$ of $G$ over $\R$ with locality $\ell$, where the dimension $t$ is minimized over all the orthogonal representations of $G$ over $\R$ with locality $\ell$.
For every set $C \subseteq [n]$, define
\[ W_C = \linspan \big ( \{u_B ~\mid~ B \in \calF~~\mbox{and}~~B \cap C = \emptyset\} \big ).\]
In words, $W_C$ is the subspace of $\R^t$ spanned by the vectors of the given orthogonal representation that correspond to vertices that are disjoint from $C$.
In particular, for a vertex $A \in \calF$, $W_A$ is the subspace spanned by the vectors associated with the neighborhood of $A$ in $G$, hence $u_A$ is orthogonal to $W_A$.
Since the dimension of a subspace spanned by the vectors of a closed neighborhood in $G$ is at most $\ell$, it follows that $\dim (W_A) \leq \ell -1$ for every $A \in \calF$.

Our goal is to prove that $\ell \geq \chi(G)$. To do so, we consider three cases, as described below.

Suppose first that for every vertex $\{i,j\} \in \calF$ there exists some $k \in [n] \setminus \{i,j\}$ for which both $\{i,k\}$ and $\{j,k\}$ are vertices of $G$ and the three vectors $u_{\{i,j\}}, u_{\{i,k\}}, u_{\{j,k\}}$ are pairwise proportional.
In this case, for every  vertex $\{i,j\} \in \calF$ associated with such an index $k$, every vertex $B \in \calF$ that is not contained in $\{i,j,k\}$ is disjoint from at least one of the sets $\{i,j\}$, $\{i,k\}$, and $\{j,k\}$, hence the vector $u_B$ is orthogonal to $u_{\{i,j\}}$ (and to $u_{\{i,k\}}$ and $u_{\{j,k\}}$ as well).
This implies that the vectors of the given orthogonal representation satisfy that every two of them are either proportional or orthogonal, hence they induce a proper coloring of the graph $K(\calF)$ with equivalence classes of proportional vectors as color classes.
This coloring satisfies that every color class is of size $3$, its locality is $\ell$, and every vertex has all colors in its closed neighborhood. This implies that $\ell \geq \chi(G)$, as required.

Suppose next that there exist distinct $i,j,k \in [n]$ for which both $\{i,j\}$ and $\{i,k\}$ are vertices of $G$ whereas $\{j,k\}$ is not, and the vectors $u_{\{i,j\}}, u_{\{i,k\}}$ are proportional.
Let $\calF' = \{ A \in \calF ~\mid~ i \notin A\}$ be the family of sets in $\calF$ that are contained in $[n] \setminus \{i\}$, and consider the subgraph $G' = K(\calF')$ of $G$.
Observe that every vertex $A \in \calF'$ is disjoint from at least one of the sets $\{i,j\}$ and $\{i,k\}$, because $i \notin A$ and $A \neq \{j,k\}$.
Since the vectors $u_{\{i,j\}}, u_{\{i,k\}}$ are proportional, for such an $A$ it holds that $u_A$ is orthogonal to $u_{\{i,j\}}$.
The restriction $(u_A)_{A \in \calF'}$ of the given orthogonal representation obviously forms an orthogonal representation of $G'$.
We turn to show that its locality is at most $\ell-1$.

Consider some vertex $A \in \calF'$, and denote by $\widetilde{W}_A$ the subspace of $\R^t$ spanned by the vectors $u_B$ with $B \in \calF'$ such that $A \cap B = \emptyset$.
Since for $B \in \calF'$, the vector $u_B$ is orthogonal to $u_{\{i,j\}}$, it follows that $u_{\{i,j\}} \notin \widetilde{W}_A$.
On the other hand, since $A$ is disjoint from at least one of the sets $\{i,j\}$ and $\{i,k\}$, it follows that $u_{\{i,j\}} \in W_A$.
The fact that $\calF' \subseteq \calF$ implies that $\widetilde{W}_A \subseteq W_A$, so we get that $\widetilde{W}_A \subsetneq W_A$, hence $\dim (\widetilde{W}_A) \leq \dim(W_A) -1 \leq \ell -2$.
This shows that the locality of the orthogonal representation $(u_A)_{A \in \calF'}$ of $G'$ is at most $\ell-1$, hence
\begin{eqnarray}\label{ineq:G'1}
\od_l(G',\R) \leq \ell-1.
\end{eqnarray}
Since the underlying ground set of the graph $G' = K(\calF')$ is of size $n-1$, we can apply the inductive assumption to obtain that
\begin{eqnarray}\label{ineq:G'2}
\chi(G') \leq \od_l(G',\R).
\end{eqnarray}
We further observe that
\begin{eqnarray}\label{ineq:G'3}
\chi(G) \leq \chi(G')+1,
\end{eqnarray}
because a proper coloring of $G'$ with $\chi(G')$ colors can be extended to a proper coloring of $G$ by assigning an additional color to all the vertices of $\calF \setminus \calF'$ which form an intersecting family.
Combining~\eqref{ineq:G'1},~\eqref{ineq:G'2}, and~\eqref{ineq:G'3}, it follows that
\[ \chi(G) - 1 \leq \chi(G') \leq \od_l (G',\R) \leq \ell -1,\]
which implies that $\ell \geq \chi(G)$, as required.

Otherwise, if none of the above two cases holds, there exists a vertex $\{i,j\} \in \calF$ such that for every $k \in [n] \setminus \{i,j\}$, either the sets $\{i,k\}$ and $\{j,k\}$ are not vertices of $\calF$, or at least one of these sets is a vertex of $\calF$ and its vector is not proportional to $u_{\{i,j\}}$.
As before, put $\calF' = \{ A \in \calF ~\mid~ i \notin A\}$, and consider the subgraph $G' = K(\calF')$ of $G$.
We turn to define an orthogonal representation $(\widetilde{u}_A)_{A \in \calF'}$ of $G'$ such that all of its vectors are orthogonal to $u_{\{i,j\}}$.
For every $A \in \calF'$ such that $j \notin A$, we simply define $\widetilde{u}_A = u_A$.
Every other vertex of $\calF'$ is of the form $\{j,k\}$ for $k \in [n] \setminus \{i,j\}$. If $u_{\{j,k\}}$ is not proportional to $u_{\{i,j\}}$ then we define $\widetilde{u}_{\{j,k\}}$ as the (nonzero) projection of $u_{\{j,k\}}$ to the subspace of $\R^t$ orthogonal to $u_{\{i,j\}}$, that is,
\[ \widetilde{u}_{\{j,k\}} = u_{\{j,k\}} - \frac{\langle u_{\{j,k\}}, u_{\{i,j\}} \rangle}{\langle u_{\{i,j\}}, u_{\{i,j\}} \rangle} \cdot  u_{\{i,j\}}.\]
Otherwise, it follows that $\{i,k\} \in \calF$ and that $u_{\{i,k\}}$ is not proportional to $u_{\{i,j\}}$, so we define $\widetilde{u}_{\{j,k\}}$ as the (nonzero) projection of $u_{\{i,k\}}$ to the subspace of $\R^t$ orthogonal to $u_{\{i,j\}}$, that is,
\[ \widetilde{u}_{\{j,k\}} = u_{\{i,k\}} - \frac{\langle u_{\{i,k\}}, u_{\{i,j\}} \rangle}{\langle u_{\{i,j\}}, u_{\{i,j\}} \rangle} \cdot u_{\{i,j\}}.\]

By definition, all the vectors $\widetilde{u}_A$ with $A \in \calF'$ are nonzero and are orthogonal to $u_{\{i,j\}}$.
We turn to show that $(\widetilde{u}_A)_{A \in \calF'}$ is an orthogonal representation of $G'$.
Consider two disjoint sets $A,B \in \calF'$.
If $j \notin A \cup B$ then we clearly have $\langle \widetilde{u}_A, \widetilde{u}_B \rangle = \langle u_A, u_B \rangle = 0$. Otherwise, without loss of generality, suppose that $A = \{j,k\}$ for some $k \in [n] \setminus \{i,j\}$. Since $B$ is disjoint from $A$, the vector $\widetilde{u}_B = u_B$ is orthogonal to the vectors $u_C$ with $C \in \{ \{i,j\}, \{i,k\},\{j,k\}\} \cap \calF$, and since $\widetilde{u}_A$ is a linear combination of them, we get that $\langle \widetilde{u}_A, \widetilde{u}_B \rangle = 0$, as required.

We next claim that the locality of the orthogonal representation $(\widetilde{u}_A)_{A \in \calF'}$ of $G'$ is at most $\ell-1$.
To prove it, it suffices to show that for every vertex $A \in \calF'$ the subspace $\widetilde{W}_A$ spanned by the vectors $\widetilde{u}_B$ with $B \in \calF'$ and $A \cap B = \emptyset$ satisfies $\dim(\widetilde{W}_A) \leq \ell-2$.
For a vertex $A \in \calF'$ such that $A \subseteq [n] \setminus\{i,j\}$, it holds that $\widetilde{W}_A \subseteq W_A$, because every vector $\widetilde{u}_B$ with $A \cap B = \emptyset$ is a linear combination of vectors $u_C$ with $C \subseteq B \cup \{i,j\}$ and thus with $A \cap C = \emptyset$.
In addition, since $u_{\{i,j\}}$ is orthogonal to $\widetilde{W}_A$, we have $u_{\{i,j\}} \in W_A \setminus \widetilde{W}_A$, implying that $\dim (\widetilde{W}_A) \leq \dim(W_A)-1 \leq \ell-2$, as required.
Every other vertex of $\calF'$, which is not contained in $[n] \setminus\{i,j\}$, is of the form $A = \{j,k\}$ for some $k \in [n] \setminus \{i,j\}$, hence it holds that $\widetilde{W}_A = W_{C}$ for $C = \{i,j,k\}$.
We turn to prove that
\begin{eqnarray}\label{eq:W_C}
\dim(W_C) \leq \ell-2,
\end{eqnarray}
which implies that $\dim(\widetilde{W}_A) \leq \ell-2$.
Observe that this will complete the proof of the theorem.
Indeed, by combining it with the above discussion it follows that $\od_l(G',\R) \leq \ell-1$.
By applying the inductive assumption to $G'$, as in the previous case, we obtain that
\[\chi(G) - 1 \leq \chi(G') \leq \od_l (G',\R) \leq \ell -1,\]
hence $\ell \geq \chi(G)$, as required.

To prove~\eqref{eq:W_C}, recall that the vectors assigned to the $2$-subsets of $C$ in $\calF$ include two vectors that are not proportional. In particular, it follows that $\calF$ includes either two or three $2$-subsets of $C$, and we consider the following two cases accordingly.

Assume first that all the three $2$-subsets of $C$ belong to $\calF$, and denote them by $A_1,A_2,A_3$.
Suppose in contradiction that $\dim (W_C) \geq \ell-1$.
For each $i \in [3]$, it follows from the definition that $W_{C} \subseteq W_{A_i}$, which by $\dim (W_{A_i}) \leq \ell -1$ implies that $W_C = W_{A_i}$. It therefore follows that $W_C = W_{A_1} = W_{A_2} = W_{A_3}$. Notice that every vertex $B \in \calF \setminus \{A_1,A_2,A_3\}$ is disjoint from at least one of $A_1,A_2,A_3$, hence its vector $u_B$ belongs to some $W_{A_i}$ and thus to $W_C$, so it is orthogonal to each of $u_{A_1}, u_{A_2}, u_{A_3}$. This implies that switching the vectors of the vertices $A_2$ and $A_3$ to $u_{A_1}$ results in another orthogonal representation of $G$, and it is not difficult to verify that its locality is at most $\ell$.
This orthogonal representation lies in a subspace of $\R^t$ whose dimension is at most $t-1$.
Indeed, $u_{A_1}, u_{A_2}, u_{A_3}$ are not pairwise proportional and they are all orthogonal to $u_B$ for every vertex $B \in \calF \setminus \{A_1,A_2,A_3\}$, so $\linspan(u_{A_2},u_{A_3})$ is not contained in the subspace of the modified representation.
By applying an orthogonal linear transformation from this subspace to $\R^{t-1}$, we get an orthogonal representation of $G$ that contradicts the minimality of $t$.

Assume next that only two of the three $2$-subsets of $C$ belong to $\calF$, denote them by $A_1,A_2$, and let $j \in [n]$ denote the element of their intersection.
As before, suppose in contradiction that $\dim (W_C) \geq \ell-1$, which implies that $W_C = W_{A_1} = W_{A_2}$.
Since one of the $2$-subsets of $C$ does not belong to $\calF$, it follows that every vertex $B \in \calF$ such that $j \notin B$ is disjoint from at least one of $A_1$ and $A_2$, hence its vector $u_B$ belongs to some $W_{A_i}$ and thus to $W_C$, so it is orthogonal to each of $u_{A_1}$ and $u_{A_2}$.
This implies that switching the vectors of all vertices $B \in \calF$ with $j \in B$ (including $A_2$) to $u_{A_1}$ results in another orthogonal representation of $G$.
Moreover, the locality of this orthogonal representation is at most $\ell$. Indeed, for a vertex that includes $j$, the vectors of its neighbors are unchanged, so the dimension of the subspace spanned by the vectors of its closed neighborhood is at most $\ell$.
Further, for a vertex that does not include $j$, the vectors of its neighbors that do not include $j$ are unchanged and are all orthogonal to $u_{A_1}$ and to $u_{A_2}$. All the other neighbors are assigned now the vector $u_{A_1}$, whereas their vectors according to the original representation include at least one of $u_{A_1}$ and $u_{A_2}$. This implies that the dimension of the subspace spanned by the vectors assigned by the modified representation to the closed neighborhood of a vertex that does not include $j$ is at most $\ell$.

The obtained orthogonal representation lies in a subspace of $\R^t$ whose dimension is at most $t-1$.
Indeed, $u_{A_1}$ and $u_{A_2}$ are not proportional and they are both orthogonal to $u_B$ for every vertex $B \in \calF$ with $j \notin B$. It follows that $u_{A_2}$ does not belong to the subspace spanned by the modified representation, hence its dimension is at most $t-1$.
By applying an orthogonal linear transformation from this subspace to $\R^{t-1}$, we again get an orthogonal representation of $G$ that contradicts the minimality of $t$, as required. This completes the proof.
\end{proof}

As an immediate corollary of Theorem~\ref{thm:xi_l_K(F)}, we obtain the following.

\begin{corollary}
For every integer $n \geq 4$,~ $\od_l(K(n,2),\R) = \od_l(S(n,2),\R) = n-2$.
\end{corollary}

\section{The Local Chromatic Number of \texorpdfstring{$K(n,3)$}{K(n,3)}}\label{sec:K(n,3)}

It was shown in~\cite{LovaszKneser} that the chromatic number of the Kneser graph $K(n,k)$ is $n-2k+2$ for all integers $n \geq 2k$.
This implies, for integers $n \geq 3k$, that the local chromatic number of $K(n,k)$ satisfies $\chi_l(K(n,k)) \geq n-3k+3$.
To see this, observe that the subgraph of $K(n,k)$ induced by the (open) neighborhood of any vertex is isomorphic to $K(n-k,k)$, whose chromatic number is $n-3k+2$.
Hence, the locality of every proper coloring of $K(n,k)$ is at least  $n-3k+3$.
For $k=3$, this shows that $\chi_l(K(n,3)) \geq n-6$ for all integers $n \geq 9$.
In this section we improve on this bound and prove that for every $n \geq 6$ it holds that $\chi_l(K(n,3)) = n-4$, confirming Theorem~\ref{thm:K(n,3)}.

\begin{proof}[ of Theorem~\ref{thm:K(n,3)}]
Every graph $G$ satisfies $\chi_l(G) \leq \chi(G)$, hence for every $n \geq 6$, it holds that
\[\chi_l(K(n,3)) \leq \chi(K(n,3)) = n-4.\]
It thus suffices to prove a matching lower bound.
We start with the two easy cases of $n \in \{6,7\}$.
The graph $K(6,3)$ is a perfect matching on $20$ vertices, and in particular it contains an edge, hence $\chi_l(K(6,3)) \geq 2$.
The chromatic number of $K(7,3)$ is $3$, so it is not bipartite, hence it contains a cycle of odd length.
It is easy to see that the local chromatic number of such a cycle is $3$, implying that $\chi_l(K(7,3)) \geq 3$.
From now on, we may and will assume that $n \geq 8$.

Let $V$ denote the vertex set of $K(n,3)$ for some $n \geq 8$, and put $\ell = \chi_l(K(n,3))$. Let $c: V \rightarrow [m]$ be a proper coloring of $K(n,3)$ with locality $\ell$, where the total number $m$ of colors is minimized over all such colorings.
Suppose for the sake of contradiction that $\ell \leq n-5$.

Consider the number $M$ of pairs $(A,j)$ of a vertex $A \in V$ and a color $j \in [m]$ for which $A$ {\em sees} the color $j$ in its (open) neighborhood, that is, there exists a vertex $B \in V$ satisfying $A \cap B = \emptyset$ and $c(B) = j$.
Since the locality of the coloring $c$ is $\ell$, it follows that every vertex sees at most $\ell-1$ colors, hence
\[M \leq (\ell-1) \cdot \tbinom{n}{3} \leq (n-6) \cdot \tbinom{n}{3}.\]
We will obtain a contradiction by proving that
\begin{eqnarray*}
M > (n-6) \cdot \tbinom{n}{3}.
\end{eqnarray*}
This inequality will be established by analyzing the number of vertices that see each of the $m$ colors of the coloring $c$.
The number of vertices that see a given color is significantly affected by the structure of its color class (which forms an independent set in $K(n,3)$).
This requires us to deal separately with color classes of several types, as described next.

Let $\mathcal{I}$ be the collection of color classes of the coloring $c: V \rightarrow [m]$.
By the definition of the graph $K(n,3)$, every $I \in \mathcal{I}$ is an intersecting family of $3$-subsets of $[n]$, and it can be seen that its size satisfies $|I| \geq 2$.
Indeed, if $I$ consists of a single vertex, then this vertex sees at most $\ell-1\leq n-6$ colors, so its color can be replaced by one of the other $m-1 \geq \chi(K(n,3))-1=n-5$ colors without increasing the locality, a contradiction to the minimality of $m$.
We consider four types of color classes $I$, referred to as~\eqref{itm:type_a},~\eqref{itm:type_b},~\eqref{itm:type_c} and~\eqref{itm:type_d}, defined as follows.
\begin{enumerate}[(a).]
  \item\label{itm:type_a} There are two elements of $[n]$ that belong to all vertices of $I$.
  \item\label{itm:type_b} $|I| =2$, and the intersection size of the two vertices of $I$ is $1$.
  \item\label{itm:type_c} $|I| \geq 3$, there is a single element of $[n]$ that belongs to all vertices of $I$, and either
    \begin{itemize}
      \item the maximum size of a union of three vertices of $I$ is at most $5$, or
      \item the maximum size of a union of three vertices of $I$ is $6$, and there is an element of $[n]$ that belongs to all but one of the vertices of $I$.
    \end{itemize}
  \item\label{itm:type_d} $I$ is not of types~\eqref{itm:type_a},~\eqref{itm:type_b} and~\eqref{itm:type_c}, that is, $|I| \geq 3$ and either
  \begin{itemize}
    \item there is a single element of $[n]$ that belongs to all vertices of $I$, the maximum size of a union of three vertices of $I$ is $6$, and there is no element of $[n]$ that belongs to all but one of the vertices of $I$, or
    \item there is a single element of $[n]$ that belongs to all vertices of $I$, and the maximum size of a union of three vertices of $I$ is $7$, or
    \item there is no element of $[n]$ that belongs to all vertices of $I$.
  \end{itemize}
\end{enumerate}

For a color class $I \in \mathcal{I}$, let $\overline{M}_I$ denote the number of vertices of $K(n,3)$ that {\em do not} see the color of $I$, i.e., the vertices that intersect each of the vertices of $I$ (including the vertices of $I$).
The following lemma provides an upper bound on $\overline{M}_I$ for each type of $I$.
(Note that Lemma~\ref{lemma:color_classes} below and Lemma~\ref{lemma:K(n,3)-J} that follows it are stated under the assumption for contradiction that $\ell \leq n-5$.)

\begin{lemma}\label{lemma:color_classes}
Let $I \in \mathcal{I}$ be a color class of the coloring $c: V \rightarrow [m]$ of $K(n,3)$. Then, the following holds.
\begin{enumerate}
  \item\label{type:a} If $I$ is of type~\eqref{itm:type_a} then $\overline{M}_I \leq \binom{n-1}{2}+\binom{n-2}{2}+(n-4)$.
  \item\label{type:b} If $I$ is of type~\eqref{itm:type_b} then $\overline{M}_I = \binom{n-1}{2}+4(n-4)$.
  \item\label{type:c} If $I$ is of type~\eqref{itm:type_c} then $\overline{M}_I \leq \binom{n-1}{2}+3(n-4)+1$.
  \item\label{type:d} If $I$ is of type~\eqref{itm:type_d} then $\overline{M}_I \leq \max \Big (7n-25, \binom{n-1}{2}+n \Big )$.
\end{enumerate}
In addition, let $G$ be a graph obtained from $K(n,3)$ by removing from its vertex set either
\begin{enumerate}[~i.]
 \item\label{remove:a} any number of color classes of type~\eqref{itm:type_a}, or
  \item\label{remove:b} a color class of type~\eqref{itm:type_b}, or
  \item\label{remove:a_b} two color classes, one of type~\eqref{itm:type_a} and one of type~\eqref{itm:type_b}, or
  \item\label{remove:c} a color class of type~\eqref{itm:type_c}.
\end{enumerate}
Then, $\chi(G) = \chi(K(n,3)) = n-4$.
\end{lemma}

\begin{remark}\label{remark:decreasing}
Note that the bounds on $\overline{M}_I$ given in the lemma for the four types of $I$ are in a decreasing order, namely, for all $n \geq 8$, it holds that
\[ \tbinom{n-1}{2}+\tbinom{n-2}{2}+(n-4) > \tbinom{n-1}{2}+4(n-4) > \tbinom{n-1}{2}+3(n-4)+1 > \max \Big (7n-25, \tbinom{n-1}{2}+n \Big ).\]
\end{remark}

Before proving Lemma~\ref{lemma:color_classes}, we show how it is applied to prove the assertion of the theorem.
Recall that to get a contradiction to the assumption $\ell \leq n-5$, it suffices to prove that
\begin{eqnarray}\label{ineq:K(n,3)}
M- (n-6) \cdot \tbinom{n}{3}  > 0.
\end{eqnarray}
The number $M$ of pairs $(A,j)$ of a vertex $A \in V$ and a color $j \in [m]$ that it sees can be written as
\begin{eqnarray}\label{def:M}
M = \sum_{I \in \mathcal{I}}{\Big (\tbinom{n}{3} - \overline{M}_I \Big )}.
\end{eqnarray}
Our strategy to prove~\eqref{ineq:K(n,3)} is the following.
If all the color classes $I \in \mathcal{I}$ are not seen by a relatively `few' vertices and thus have `small' values of $\overline{M}_I$ (e.g., when all of them are of type~\eqref{itm:type_d}), then it can be verified, using~\eqref{def:M}, that~\eqref{ineq:K(n,3)} holds. However, for the case where some of the color classes $I \in \mathcal{I}$ have `large' values of $\overline{M}_I$, we will provide a lower bound on the number $m$ of the colors of $c$, which will again allow us to derive the required inequality.
For such a lower bound on the number of colors, we will need the following lemma.

\begin{lemma}\label{lemma:K(n,3)-J}
For an integer $r \geq 0$, let $J \subseteq V$ be a union of $r$ color classes of the coloring $c: V \rightarrow [m]$, such that the graph $G$ obtained from $K(n,3)$ by removing the vertices of $J$ satisfies $\chi(G) = n-4$. Then, $m \geq r + n-3$.
\end{lemma}

\begin{proof}
The restriction of the coloring $c$ to the vertex set $V \setminus J$ of $G$ forms a proper coloring of $G$ with locality at most $\ell$ and with $m-r$ colors.
By assumption, we have $\ell \leq n-5 < n-4 = \chi(G)$.
However, every proper coloring of a graph $G$ with locality strictly smaller than $\chi(G)$ must use more than $\chi(G)$ colors, since otherwise one could avoid one of the colors by assigning to each vertex of that color some color different from those of its neighbors, resulting in a proper coloring of $G$ with less than $\chi(G)$ colors.
It thus follows that the number of colors used by $c$ in $G$ is larger than its chromatic number, that is, $m -r > n-4$, yielding that $m \geq r +n-3$, as desired.
\end{proof}

Equipped with Lemmas~\ref{lemma:color_classes} and~\ref{lemma:K(n,3)-J}, we prove~\eqref{ineq:K(n,3)} by considering the following four cases.
\begin{itemize}
  \item Suppose that there are $r \geq 2$ color classes in $\mathcal{I}$ of type~\eqref{itm:type_a}. By Item~\ref{remove:a} of Lemma~\ref{lemma:color_classes}, the chromatic number of the graph obtained from $K(n,3)$ by removing these $r$ color classes is $n-4$.
      By Lemma~\ref{lemma:K(n,3)-J}, we have $m \geq r+n-3$, so there are at least $n-3$ color classes in $\mathcal{I}$ of types~\eqref{itm:type_b},~\eqref{itm:type_c} or~\eqref{itm:type_d}.
      This implies, using Lemma~\ref{lemma:color_classes}, that there are $r \geq 2$ color classes $I \in \mathcal{I}$ that satisfy $\overline{M}_I \leq \binom{n-1}{2}+\binom{n-2}{2}+(n-4)$ and at least $n-3$ other color classes $I \in \mathcal{I}$ that satisfy $\overline{M}_I \leq \binom{n-1}{2}+4(n-4)$ (see Remark~\ref{remark:decreasing}). By~\eqref{def:M}, this implies that
      \begin{eqnarray*}
        M - (n-6) \cdot \tbinom{n}{3} &\geq& 2 \cdot \Big (\tbinom{n}{3} - \big [\tbinom{n-1}{2}+\tbinom{n-2}{2}+(n-4) \big ] \Big )\\
         &+& (n-3) \cdot \Big  (\tbinom{n}{3} - \big [ \tbinom{n-1}{2}+4(n-4) \big ] \Big ) - (n-6) \cdot \tbinom{n}{3}  \\
         &=& \tfrac{n^3}{3} - \tfrac{11n^2}{2} + \tfrac{181n}{6}-45.
      \end{eqnarray*}
      A simple calculation shows that the above is positive for all $n \geq 8$, as required for~\eqref{ineq:K(n,3)}.
  \item Suppose next that there is exactly one color class in $\mathcal{I}$ of type~\eqref{itm:type_a}. \\
  If there is at least one color class in $\mathcal{I}$ of type~\eqref{itm:type_b}, then by Item~\ref{remove:a_b} of Lemma~\ref{lemma:color_classes}, the chromatic number of the graph obtained from $K(n,3)$ by removing the color class of type~\eqref{itm:type_a} and one color class of type~\eqref{itm:type_b} is $n-4$.
      By Lemma~\ref{lemma:K(n,3)-J}, there are at least $n-3$ other color classes in $\mathcal{I}$ of types~\eqref{itm:type_b},~\eqref{itm:type_c} or~\eqref{itm:type_d}.
      This implies, using Lemma~\ref{lemma:color_classes}, that there is a color class $I \in \mathcal{I}$ that satisfies $\overline{M}_I \leq \binom{n-1}{2}+\binom{n-2}{2}+(n-4)$, and there are at least $n-2$ other color classes $I \in \mathcal{I}$ that satisfy $\overline{M}_I \leq \binom{n-1}{2}+4(n-4)$ (see Remark~\ref{remark:decreasing}). By~\eqref{def:M}, this implies that
      \begin{eqnarray*}
        M - (n-6) \cdot \tbinom{n}{3} &\geq& 1 \cdot \Big ( \tbinom{n}{3} - \big [\tbinom{n-1}{2}+\tbinom{n-2}{2}+(n-4) \big ] \Big )\\
         &+& (n-2) \cdot \Big  (\tbinom{n}{3} - \big [\tbinom{n-1}{2}+4(n-4) \big ] \Big ) - (n-6) \cdot \tbinom{n}{3}.
      \end{eqnarray*}
      The above can be only larger than what we have got in the previous case, so we are done.

      If, however, there is no color class in $\mathcal{I}$ of type~\eqref{itm:type_b}, then by Item~\ref{remove:a} of Lemma~\ref{lemma:color_classes}, the chromatic number of the graph obtained from $K(n,3)$ by removing the single color class of type~\eqref{itm:type_a} is $n-4$.
      By Lemma~\ref{lemma:K(n,3)-J}, there are at least $n-3$ color classes in $\mathcal{I}$ of types~\eqref{itm:type_c} or~\eqref{itm:type_d}.
      Combining~\eqref{def:M} with Lemma~\ref{lemma:color_classes} (see Remark~\ref{remark:decreasing}), we obtain that
      \begin{eqnarray*}
        M - (n-6) \cdot \tbinom{n}{3} &\geq& 1 \cdot \Big (\tbinom{n}{3} - \big [\tbinom{n-1}{2}+\tbinom{n-2}{2}+(n-4) \big ] \Big )\\
         &+& (n-3) \cdot \Big  (\tbinom{n}{3} - \big [\tbinom{n-1}{2}+3(n-4) +1 \big] \Big ) - (n-6) \cdot \tbinom{n}{3} \\
         &=& \tfrac{n^3}{6} -3 n^2 +\tfrac{113n}{6}-30.
      \end{eqnarray*}
      A simple calculation shows that the above is positive for all $n \geq 8$, as required for~\eqref{ineq:K(n,3)}.

  \item Suppose next that no color class of $\mathcal{I}$ is of type~\eqref{itm:type_a} and that there is at least one color class of type~\eqref{itm:type_b} or~\eqref{itm:type_c}.
      By Items~\ref{remove:b} and~\ref{remove:c} of Lemma~\ref{lemma:color_classes}, the chromatic number of the graph obtained from $K(n,3)$ by removing the vertices of one such color class is $n-4$.
      By Lemma~\ref{lemma:K(n,3)-J}, there are at least $n-2$ color classes in $\mathcal{I}$.
      Combining~\eqref{def:M} with Lemma~\ref{lemma:color_classes} (see Remark~\ref{remark:decreasing}), we obtain that
      \begin{eqnarray*}
        M - (n-6) \cdot \tbinom{n}{3} &\geq& (n-2) \cdot \Big  (\tbinom{n}{3} - \big [\tbinom{n-1}{2}+4(n-4) \big] \Big ) - (n-6) \cdot \tbinom{n}{3}\\
        & = &  \tfrac{1}{6} \cdot (n-2)(n-9)(n-10).
      \end{eqnarray*}
      The obtained bound is clearly non-negative for all integers $n \geq 8$.
      In fact, it can be seen that the above inequality is strict for every such $n$. To see this, observe that either the number of color classes in $\mathcal{I}$ is strictly larger than $n-2$ or at least one of them is not of type~\eqref{itm:type_b}, because every color class of this type includes only two vertices, whereas for all $n \geq 8$ it holds that $\binom{n}{3} > 2 \cdot (n-2)$.
      We therefore derive the required inequality~\eqref{ineq:K(n,3)}.

  \item We are left with the case where all the color classes of $\mathcal{I}$ are of type~\eqref{itm:type_d}. By Lemma~\ref{lemma:K(n,3)-J}, applied with $r=0$, the number of color classes of $\mathcal{I}$ is at least $n-3$. Combining~\eqref{def:M} with Lemma~\ref{lemma:color_classes}, we obtain that
      \begin{eqnarray*}
        M - (n-6) \cdot \tbinom{n}{3} &\geq& (n-3) \cdot \big  (\tbinom{n}{3} - \max \big (7n-25,~\tbinom{n-1}{2}+n \big) \big ) - (n-6) \cdot \tbinom{n}{3}\\
        &=& \tfrac{1}{2} \cdot \min \big ( n^2-3n+6,~ n^3-17n^2+94n-150 \big).
      \end{eqnarray*}
      A simple calculation shows that the above is positive for all $n \geq 8$, as required for~\eqref{ineq:K(n,3)}.
\end{itemize}

It remains to prove Lemma~\ref{lemma:color_classes}.

\begin{proof}[ of Lemma~\ref{lemma:color_classes}]
For Item~\ref{type:a}, let $I$ be a color class of type~\eqref{itm:type_a}. Since $|I| \geq 2$, $I$ includes two vertices of the form $A_1 = \{i_1,i_2,i_3\}$, $A_2 = \{i_1,i_2,i_4\}$ for distinct $i_1, \ldots, i_4 \in [n]$, where all the vertices of $I$ include both $i_1$ and $i_2$. The vertices of $K(n,3)$ that do not see the color class of $I$ must intersect both $A_1$ and $A_2$ and thus must include $i_1$ or $i_2$ or both $i_3$ and $i_4$. It thus follows that
\[\overline{M}_I \leq \tbinom{n-1}{2}+ \tbinom{n-2}{2}+(n-4).\]
We turn to prove Item~\ref{remove:a}.
Suppose that the coloring $c$ has $r$ color classes of type~\eqref{itm:type_a}. The vertices of every such color class share a pair of common elements. Observe that these $r$ pairs are pairwise disjoint, because otherwise one could merge two color classes into one to obtain a proper coloring with fewer colors and yet locality at most $\ell$, in contradiction to the minimality of $m$.
Now, notice that the graph $K(n,3)$ is isomorphic to the one obtained by applying a permutation of $[n]$ to the elements of its vertices. Recall that the Schrijver graph $S(n,3)$ is the subgraph of $K(n,3)$ induced by the vertices that include no two consecutive elements modulo $n$.
By considering a permutation of $[n]$ that maps each of the pairwise disjoint $r$ pairs to consecutive elements modulo $n$, it follows that the graph obtained from $K(n,3)$ by removing the color classes of type~\eqref{itm:type_a} contains a subgraph isomorphic to $S(n,3)$. Since the chromatic number of the latter is $n-4$, the removal of the vertices does not change the chromatic number, as required.

For Item~\ref{type:b}, let $I$ be a color class of type~\eqref{itm:type_b}, so $I$ includes precisely two vertices of the form $A_1 = \{i_1,i_2,i_3\}$ and $A_2 = \{i_1,i_4,i_5\}$ for distinct $i_1, \ldots, i_5 \in [n]$. The vertices that do not see the color of $I$ are those that include $i_1$ or intersect both $\{i_2,i_3\}$ and $\{i_4,i_5\}$. Hence, \[\overline{M}_I = \tbinom{n-1}{2} +2 \cdot 2 \cdot (n-5) +4 = \tbinom{n-1}{2} +4(n-4).\]
To prove Item~\ref{remove:b}, observe that by applying a permutation of $[n]$ to the elements of the vertices of $K(n,3)$, one can ensure that each of $A_1$ and $A_2$ includes consecutive elements (say, by a permutation that maps $i_j$ to $j$ for all $j \in [5]$). Hence, by removing the vertices of $I$ from $K(n,3)$ we get a graph that contains a subgraph isomorphic to $S(n,3)$, so its chromatic number is $n-4$.

For Item~\ref{remove:a_b}, we claim that given two color classes of $c$, $I_1$ of type~\eqref{itm:type_a} and $I_2$ of type~\eqref{itm:type_b}, the graph obtained from $K(n,3)$ by removing the vertices of $I_1 \cup I_2$ has chromatic number $n-4$. To see this, denote by $i_1$ and $i_2$ the two common elements of the vertices of $I_1$, and denote by $\{i_3, i_4, i_5\}$ and $\{i_3, i_6, i_7\}$ the vertices of $I_2$ which intersect only at $i_3$. Note that none of $i_1$ and $i_2$ belongs to both these sets, and thus $i_3 \notin \{i_1, i_2\}$, since otherwise one could merge the color classes $I_1$ and $I_2$ into one to obtain a proper coloring with fewer colors and yet locality at most $\ell$, in contradiction to the minimality of $m$.
Hence, it can be assumed, without loss of generality, that $i_4,i_6,i_7$ are all different from $i_1$.
Applying a permutation of $[n]$ that maps $i_1$ to $1$, $i_2$ to $2$, and $i_j$ to $j$ for each $j \in [7]$ such that $i_j \notin \{i_1,i_2\}$, it follows that all the vertices of $I_1 \cup I_2$ include consecutive elements, hence by removing them from $K(n,3)$ we still have a subgraph isomorphic to $S(n,3)$, so the chromatic number is $n-4$.

For Items~\ref{type:c} and~\ref{remove:c}, let $I$ be a color class of type~\eqref{itm:type_c}, so $|I| \geq 3$ and there exists a single element $i_1 \in [n]$ that belongs to all of its vertices. Consider the following three possibilities.
\begin{itemize}
  \item Suppose that the maximum size of a union of three vertices of $I$ is $4$.
  Here, every two vertices of $I$ intersect at two elements, which implies that $I$ includes precisely three vertices of the form $A_1 = \{i_1,i_2,i_3\}$, $A_2 = \{i_1,i_2,i_4\}$, and $A_3 = \{i_1,i_3,i_4\}$ for distinct $i_1, \ldots, i_4 \in [n]$. The vertices that do not see the color of $I$ are those that include $i_1$ or at least two elements of $\{i_2,i_3,i_4\}$, hence $\overline{M}_I = \binom{n-1}{2}+3(n-4)+1$. By applying a permutation of $[n]$ that maps $i_j$ to $j$ for all $j \in [4]$, it follows that all the vertices of $I$ include consecutive elements, hence by removing them from $K(n,3)$ we still have a subgraph isomorphic to $S(n,3)$, so the chromatic number is $n-4$.

  \item Suppose that the maximum size of a union of three vertices of $I$ is $5$.
  First, observe that $I$ includes two vertices with intersection size $1$, $A_1 = \{i_1,i_2,i_3\}$ and $A_2 = \{i_1,i_4,i_5\}$ for distinct $i_1, \ldots, i_5 \in [n]$, because as above, if every two vertices of $I$ intersect at two elements then the total number of elements in the vertices of $I$ cannot exceed $4$.
  Since the maximum size of a union of three vertices of $I$ is $5$, using $|A_1 \cup A_2|=5$, the vertices of $I$ cannot include any element outside of $A_1 \cup A_2$. By $|I| \geq 3$, $I$ includes, without loss of generality, a vertex of the form $A_3 = \{i_1, i_2, i_4\}$. Other than that, it might include some of the other subsets of $\{i_1,\ldots,i_5\}$ that include $i_1$, i.e., $\{i_1, i_2, i_5\}$, $\{i_1, i_3, i_4\}$, and $\{i_1, i_3, i_5\}$.
  The vertices that do not see the color of $I$ must intersect each of $A_1$, $A_2$, and $A_3$, hence they must include $i_1$ or intersect each of $\{i_2,i_3\}$, $\{i_4,i_5\}$ and $\{i_2,i_4\}$. This implies that
  \[\overline{M}_I \leq \tbinom{n-1}{2}+3(n-4)+1.\]
  By applying a permutation of $[n]$ that maps $i_j$ to $j$ for all $j \in [4]$ and $i_5$ to $n$, it follows that all the vertices of $I$ include consecutive elements modulo $n$, hence by removing them from $K(n,3)$ we still have a subgraph isomorphic to $S(n,3)$, so the chromatic number is $n-4$.

  \item Now, suppose that the maximum size of a union of three vertices of $I$ is $6$, and that there is an element of $[n]$ that belongs to all but one of the vertices of $I$.
      As in the previous case, $I$ includes two vertices with intersection size $1$, $A_1 = \{i_1,i_2,i_3\}$ and $A_2 = \{i_1,i_4,i_5\}$ for distinct $i_1, \ldots, i_5 \in [n]$.
      Since the maximum size of a union of three vertices of $I$ is $6$, its vertices should include an additional element $i_6 \in [n]$ different from $i_1, \ldots, i_5 \in [n]$, so without loss of generality suppose that it includes the vertex $A_3 = \{i_1, i_2, i_6\}$.
      By assumption, it follows that all the other vertices of $I$ contain $\{i_1,i_2\}$.
      Since the vertices that do not see the color of $I$ intersect each of $A_1$, $A_2$, and $A_3$, their number satisfies
      \[\overline{M}_I \leq \tbinom{n-1}{2}+2(n-4)+1+2 = \tbinom{n-1}{2}+2n-5 \leq \tbinom{n-1}{2}+3(n-4)+1,\]
      where the last inequality clearly holds for all $n \geq 8$.
      By applying a permutation of $[n]$ that maps $i_j$ to $j$ for all $j \in [5]$, it follows, as before, that the graph obtained from $K(n,3)$ by removing the vertices of $I$ has chromatic umber $n-4$.
\end{itemize}

Finally, for Item~\ref{type:d}, let $I$ be a color class of type~\eqref{itm:type_d}. It satisfies $|I| \geq 3$, and we consider the following three possibilities.
\begin{itemize}
  \item Suppose that there exists a single element $i_1 \in [n]$ that belongs to all vertices of $I$, that the maximum size of a union of three vertices of $I$ is $6$, and that there is no element of $[n]$ that belongs to all but one of the vertices of $I$.
      As explained in the previous case, $I$ includes three vertices of the form $A_1 = \{i_1,i_2,i_3\}$, $A_2 = \{i_1,i_4,i_5\}$, and $A_3 = \{i_1, i_2, i_6\}$ for distinct $i_1, \ldots, i_6 \in [n]$.
      Since there is no element of $[n]$ that belongs to all but one of the vertices of $I$, we must have another vertex in $I$ that does not include $i_2$. This vertex includes $i_1$ and since the size of its union with $A_1 \cup A_2$ cannot exceed $6$, it must include an element from $\{i_3,i_4,i_5\}$. Without loss of generality, we consider the following three cases.
      \begin{itemize}
        \item If $I$ includes the vertex $A_4 = \{i_1, i_3, i_4\}$ then the number of vertices that do not see the color of $I$, and in particular intersect each of $A_1$, $A_2$, $A_3$, and $A_4$, satisfies \[\overline{M}_I \leq \tbinom{n-1}{2}+(n-3)+1+2= \tbinom{n-1}{2}+n.\]
        \item If $I$ includes a vertex of the form $A_4 = \{i_1, i_3, i_7\}$ for $i_7 \notin \{i_1, \ldots, i_5\}$ then we must have $i_6 = i_7$, because otherwise we would have $|A_2 \cup A_3 \cup A_4|=7$.
            The vertices that do not see the color of $I$ intersect each of $A_1$, $A_2$, $A_3$, and $A_4$, hence
            \[\overline{M}_I \leq \tbinom{n-1}{2}+4+2= \tbinom{n-1}{2}+6 \leq \tbinom{n-1}{2}+n.\]
        \item If $I$ includes a vertex of the form $A_4 = \{i_1, i_4, i_7\}$ for $i_7 \notin \{i_1, \ldots, i_5\}$, where $i_7$ might and might not be equal to $i_6$, then
            it follows that
            \[\overline{M}_I \leq \tbinom{n-1}{2}+(n-3)+1+2 \leq \tbinom{n-1}{2}+n.\]
      \end{itemize}

  \item Suppose that there exists a single element $i_1 \in [n]$ that belongs to all vertices of $I$, and that the maximum size of a union of three vertices of $I$ is $7$.
  Here, $I$ includes three vertices of the form $A_1 = \{i_1,i_2,i_3\}$, $A_2 = \{i_1,i_4,i_5\}$, and $A_3 = \{i_1,i_6,i_7\}$ for distinct $i_1, \ldots, i_7 \in [n]$.
  The vertices that do not see the color of $I$ must include $i_1$ or intersect each of $\{i_2,i_3\}$, $\{i_4,i_5\}$, and $\{i_6,i_7\}$. Hence, their number satisfies $\overline{M}_I \leq \binom{n-1}{2} + 8 \leq \binom{n-1}{2} + n$.

  \item Suppose that the vertices of $I$ do not share any common element. We consider the following two cases.
  \begin{itemize}
    \item If there is a pair of vertices in $I$ with intersection size $1$, denote such vertices by $A_1 = \{i_1, i_2, i_3\}$ and $A_2 = \{i_1, i_4, i_5\}$ for distinct $i_1, \ldots, i_5 \in [n]$. The collection $I$ must include a set that does not include $i_1$, which can be denoted, without loss of generality, by $A_3 = \{i_2, i_4, i_6\}$ for $i_6 \notin \{i_1,i_2,i_4\}$. Every vertex of $K(n,3)$ that does not see the color of $I$ must intersect each of the sets $A_1,A_2,A_3$.
        The number of vertices that do not see the color of $I$ and include $i_1$ is at most $3n-9$; the number of those that do not include $i_1$ but include $i_2$ is at most $2n-7$; the number of those that do not include $i_1$ and $i_2$ but include $i_4$ is at most $n-4$. And, finally, we should count the vertices that do not include $i_1$, $i_2$, $i_4$ and yet do not see the color of $I$. If $i_6 \notin \{i_3,i_5\}$ then there is at most one such vertex, and otherwise there are at most $n-5$ such vertices. This implies that $\overline{M}_I \leq (3n-9)+(2n-7)+(n-4)+\max(1,n-5) = 7n-25$, as required.
    \item If there is no pair of vertices in $I$ with intersection size $1$, then it includes two vertices of the form $A_1 = \{i_1, i_2, i_3\}$ and $A_2 = \{i_1, i_2, i_4\}$ for distinct $i_1, \ldots, i_4 \in [n]$, and since there is no element that is common to all the vertices of $I$, it has to include the vertices $A_3 = \{i_2, i_3, i_4\}$ and $A_4 = \{i_1, i_3, i_4\}$ as well. A vertex of $K(n,3)$ that does not see the color of $I$ must intersect each of these four vertices, and thus must include at least two of the elements of $\{i_1,i_2,i_3,i_4\}$. This implies that $\overline{M}_I \leq \binom{4}{2} \cdot (n-4) +4 = 6n-20 \leq 7n-25$.
  \end{itemize}
\end{itemize}
Therefore, for every color class $I$ of type~\eqref{itm:type_d}, it holds that $\overline{M}_I \leq \max \Big (7n-25, \binom{n-1}{2}+n \Big )$.
\end{proof}
The proof of the theorem is completed.
\end{proof}

\section{Hardness of Determining the Local Orthogonality Dimension}\label{sec:hardness}

In this section we consider, for any integer $k$ and a field $\Fset$, the computational problem of deciding whether an input graph $G$ satisfies $\od_l(G,\Fset) \leq k$.
We first observe that the problem is tractable for $k \in \{1,2\}$ and then prove our $\NP$-hardness result for $k \geq 3$, confirming Theorem~\ref{thm:hardness}.

\subsection{Local Orthogonality Dimension at Most Two}
Notice, first, that for a graph $G$ and a field $\Fset$, $\od_l(G,\Fset) = 1$ if and only if $G$ is edgeless.
The following simple lemma characterizes the graphs $G$ satisfying $\od_l(G,\Fset) \leq 2$.

\begin{lemma}\label{lemma:od<=2}
For every graph $G$ and a field $\Fset$, $\chi(G) \leq 2$ if and only if $\od_l(G,\Fset) \leq 2$.
\end{lemma}

\begin{proof}
Let $G$ be a graph and let $\Fset$ be a field.
If $\chi(G) \leq 2$ then it easily follows, by~\eqref{eq:od_l_vs_chi_l}, that $\od_l(G,\Fset) \leq \chi_l(G) \leq \chi(G) \leq 2$.
For the other direction, suppose that $\chi(G) > 2$, which implies that $G$ contains a cycle of odd length.
To complete the proof, it suffices to show that for every odd integer $r \geq 3$, it holds that $\od_l(C_r,\Fset) >2$, where $C_r$ is the cycle on $r$ vertices.

Suppose in contradiction that there exists an orthogonal representation of $C_r$ with locality at most $2$ assigning the vectors $u_1,\dots , u_r$ of $\Fset^t$ to its vertices along the cycle.
Put $W = \linspan(u_1,u_2)$.
Since $u_1$ and $u_2$ are not self-orthogonal and satisfy $\langle u_1, u_2 \rangle = 0$, it follows that $\dim(W) = 2$. By locality, the subspace $\linspan(u_1,u_2,u_3)$ spanned by the vectors of the closed neighborhood of the second vertex has dimension $2$, hence $u_3 \in W$.
Write $u_3 = \alpha \cdot u_1 + \beta \cdot u_2$ for some $\alpha,\beta \in \Fset$, and consider the inner product of both sides of this equality with $u_2$.
Since $\langle u_1, u_2 \rangle = \langle u_2, u_3 \rangle = 0$ and $\langle u_2, u_2 \rangle \neq 0$, it follows that $\beta = 0$, hence $u_1$ and $u_3$ are proportional, that is, there exists $\alpha \in \Fset$ for which $u_3 = \alpha \cdot u_1$. Similarly, by considering the closed neighborhood of the third vertex, it follows that $u_2$ and $u_4$ are proportional. Proceeding this way, it follows that all the vectors $u_i$ with odd $i$ are proportional and all the vectors $u_i$ with even $i$ are proportional.
In particular, since $r$ is odd, the vectors $u_1$ and $u_r$ are proportional, in contradiction to the fact that they are orthogonal but are not self-orthogonal.
\end{proof}

As is well known, the $2$-colorability problem can be decided in polynomial running-time. We thus derive from Lemma~\ref{lemma:od<=2} the following.

\begin{proposition}\label{prop:k=1,2_easy}
For every $k \in \{1,2\}$ and for every field $\Fset$, it is possible to decide whether an input graph $G$ satisfies $\xi_l(G,\mathbb{F}) \leq k$ in polynomial running-time.
\end{proposition}

\subsection{Hardness Proof}

We turn to prove now that for every field $\Fset$, it is $\NP$-hard to decide whether an input graph $G$ satisfies $\od_l(G,\Fset) \leq 3$.
This is shown by a reduction from the standard satisfiability problem $\SAT$.
The reduction, presented in Section~\ref{sec:reduction} below, combines the reductions that were used by Peeters~\cite{Peeters96} and by Osang~\cite{Osang} to prove the hardness of determining, respectively, the orthogonality dimension and the local chromatic number.
After proving the completeness and soundness of the reduction, in Sections~\ref{sec:complete} and~\ref{sec:sound} respectively, we derive Theorem~\ref{thm:hardness} in Section~\ref{sec:all_together}.

\subsubsection{The Reduction}\label{sec:reduction}

Our reduction from $\SAT$ to the problem of deciding whether the local orthogonality dimension of a graph over any field $\Fset$ is at most $3$ is performed in two steps.

Let $\varphi$ be an instance of $\SAT$, i.e., a CNF formula, and denote its variables by $x_1, \ldots , x_k$.
In the first step of the reduction, we construct a graph $G$ as follows.
The graph $G$ consists of two adjacent vertices for every variable $x_i$, representing its two literals $x_i$ and $\overline{x_i}$.
These $2k$ vertices are connected to another vertex denoted by $w$.
We connect this vertex $w$ to two other vertices denoted by $t$ and $f$ and connect them by an edge (see Figure~\ref{fig:reduction}).
Then, for every clause of the formula $\varphi$, we construct a chain of {\em OR gadgets}. The OR gadget, given in Figure~\ref{fig:or}, consists of a triangle, where one of its vertices is adjacent to the vertex $w$ and is referred to as the {\em top vertex} of the gadget, and the other two are adjacent to some other two vertices that are referred to as the {\em base vertices} of the gadget.
The base vertices of the first OR gadget of a given clause are those of its first two literals.
The base vertices of the second OR gadget are the top vertex of the first OR gadget and the vertex of the third literal of the clause.
Similarly, the base vertices of the third OR gadget are the top vertex of the second OR gadget and the vertex of the fourth literal of the clause.
Proceeding this way, we get from a clause that consists of $r$ literals, a chain of $r-1$ OR gadgets, and we identify the top vertex of the last one with the vertex $t$.
Note that the top vertex of every OR gadget is adjacent to the vertex $w$.

Next, in the second step of the reduction, we construct a graph $G'$ from $G$ by adding a gadget, denoted by $H_{i,j}$, for each pair of vertices $i \in \{w,t,f\}$ and $j \notin \{w,t,f\}$ in $G$.
The graph $H_{i,j}$ is defined as a union of two triangles, whose vertices are $\{i, a_{i,j}, b_{i,j}\}$ and $\{j,d_{i,j}, c_{i,j}\}$, with a matching connecting the vertices $i, a_{i,j},b_{i,j}$ to the vertices $d_{i,j}, j, c_{i,j}$ respectively (see Figure~\ref{fig:Hij}). Note that the two vertices $i,j$ of $G$ are identified with the two vertices $i,j$ of their gadget $H_{i,j}$ in $G'$. This completes the description of the output $G'$ of the reduction, which can clearly be constructed in polynomial running-time.

\begin{figure}[!htb]
\minipage{0.27\textwidth}
\centering
  \includegraphics[width=1.00\textwidth]{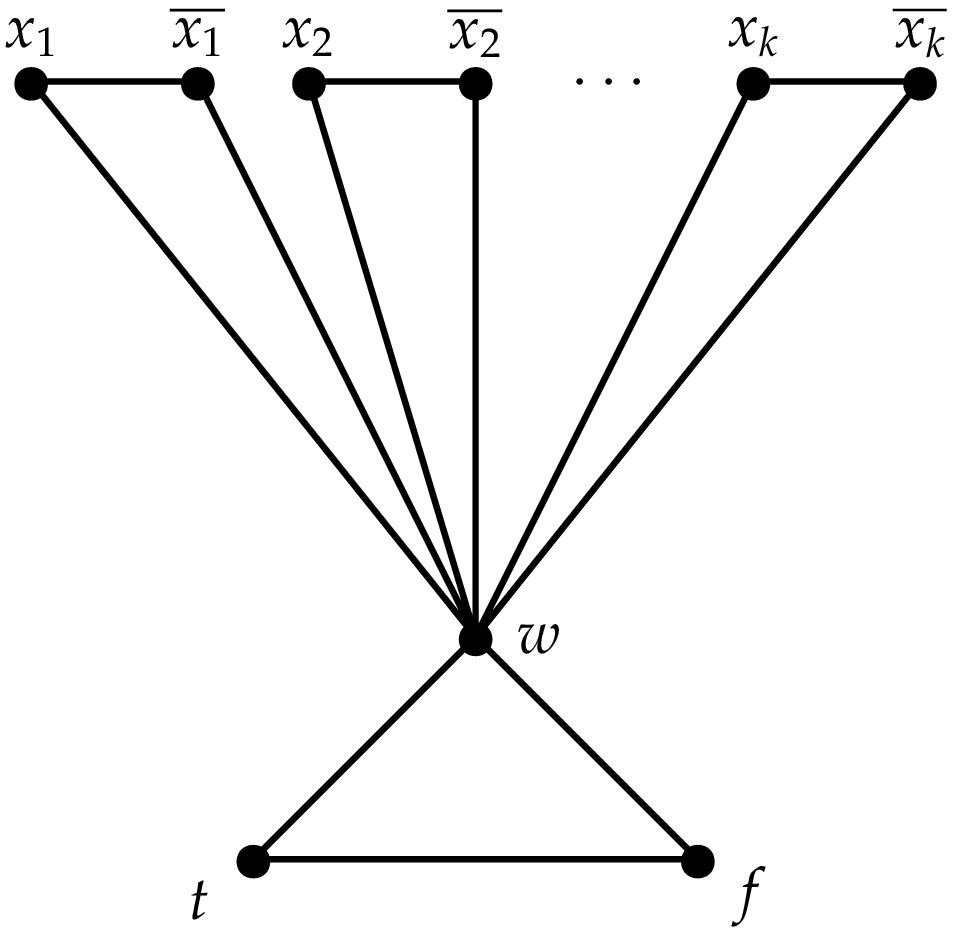}
    \caption{Basic gadget}
	\label{fig:reduction}
\endminipage\hfill
\minipage{0.27\textwidth}
\centering
  \includegraphics[width=0.85\textwidth]{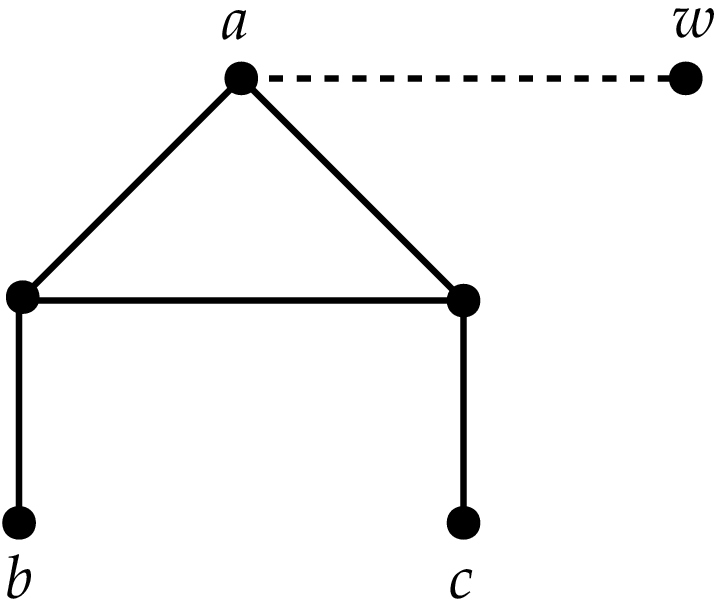}
    \caption{OR gadget \newline The top vertex is $a$; the base vertices are $b$ and $c$.}
	\label{fig:or}
\endminipage\hfill
\minipage{0.27\textwidth}
\centering
    \includegraphics[width=0.95\textwidth]{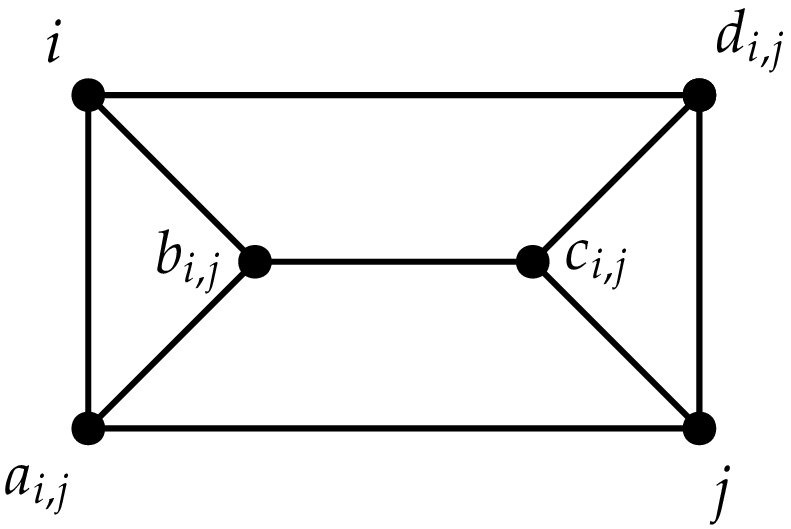}
    \caption{$H_{i,j}$ gadget}
	\label{fig:Hij}
\endminipage
\end{figure}

\subsubsection{Completeness}\label{sec:complete}
The completeness of the reduction is proved via the following two lemmas.

\begin{lemma}\label{lemma:complete1}
If $\varphi$ is satisfiable then $\chi(G) \leq 3$.
\end{lemma}

\begin{proof}
If $\varphi$ is satisfiable then there exists an assignment to its variables that satisfies all of its clauses, that is, every clause has a literal that is assigned `true'.
The constructed graph $G$ can be properly colored using the colors $\{w,t,f\}$ as follows.
We first color the vertices $w,t,f$ by the colors $w,t,f$ respectively. Next, we color every vertex representing a literal ($x_i$ or $\overline{x_i}$, $i \in [k]$) by $t$ or $f$ according to whether its truth value is `true' or `false' respectively. Now, observe that the OR gadget satisfies that given arbitrary colors from $\{t,f\}$ for its base vertices, one can extend the coloring properly so that the top vertex is colored either $t$ or $f$. Moreover, if at least one of the base vertices is colored $t$, then it is possible to properly extend the coloring so that the top vertex is colored $t$ as well. By construction, since the given assignment satisfies the formula $\varphi$, the coloring of the graph can be properly extended, so that the top vertex of the last OR gadget of every clause, which is identified with the vertex $t$, is colored $t$. This implies that $\chi(G) \leq 3$, as required.
\end{proof}

\begin{lemma}\label{lemma:complete2}
If $\chi(G) \leq 3$ then $\chi(G') \leq 3$.
\end{lemma}

\begin{proof}
Suppose that $\chi(G) \leq 3$. Recall that the graph $G'$ is obtained from $G$ by adding the gadget $H_{i,j}$ for certain pairs of vertices $i,j$ of $G$. Hence, to prove that $\chi(G') \leq 3$, it suffices to show that the gadget $H_{i,j}$ has a proper $3$-coloring where the vertices $i$ and $j$ share the same color and that it has a proper $3$-coloring where they do not. The former follows from the color classes $\{ \{i,j\} ,\{a_{i,j},c_{i,j}\},\{b_{i,j},d_{i,j}\} \}$ and the latter from the color classes $\{ \{i,c_{i,j}\},\{a_{i,j},d_{i,j}\}, \{j,b_{i,j}\} \}$, so we are done.
\end{proof}

The following corollary summarizes the completeness of the reduction.
\begin{corollary}\label{cor:complete}
For every field $\Fset$, if $\varphi$ is satisfiable then $\od_l(G',\Fset) \leq 3$.
\end{corollary}

\begin{proof}
By Lemmas~\ref{lemma:complete1} and~\ref{lemma:complete2}, it follows that if $\varphi$ is satisfiable then $\chi(G') \leq 3$.
In particular, for every field $\Fset$, it holds by~\eqref{eq:od_l_vs_chi_l} that $\od_l(G',\Fset) \leq \chi_l(G') \leq \chi(G') \leq 3$, as required.
\end{proof}

\subsubsection{Soundness}\label{sec:sound}

The following lemma provides a crucial property of the gadget graph $H_{i,j}$.

\begin{lemma}\label{lemma:Hij}
Let $H = H_{i,j}=(V,E)$ be the graph given in Figure~\ref{fig:Hij}, and let $\Fset$ be a field.
For an integer $t$, let $U$ be a subspace of $\Fset^t$ of dimension $\dim(U)=3$, and let $(u_v)_{v \in V}$ be an orthogonal representation of $H$ over $\Fset$ such that $u_v \in U$ for every $v \in V$.
Then, the vectors $u_i$ and $u_j$ of the vertices $i$ and $j$ of $H$ satisfy either $\langle u_i,u_j \rangle = 0$ or $u_i = \alpha \cdot u_j$ for some $\alpha \in \Fset$.
\end{lemma}

\begin{proof}
Consider the orthogonal representation $(u_v)_{v \in V}$ of $H$ over $\Fset$ given in the lemma.
The vectors of the vertices $i, a_{i,j}, b_{i,j}$ are pairwise orthogonal over $\Fset$, and since they are not self-orthogonal, it follows that they are linearly independent over $\Fset$. By the assumption $\dim(U)=3$, this implies that $U = \linspan(u_i, u_{a_{i,j}}, u_{b_{i,j}})$.
It thus follows that they span each of the vectors of the other vertices $j, c_{i,j}, d_{i,j}$ of $H$.
Write
\[u_j = \alpha_1 \cdot u_i + \alpha_2 \cdot u_{b_{i,j}} + \alpha_3 \cdot u_{a_{i,j}}\]
for some coefficients $\alpha_1, \alpha_2, \alpha_3 \in \Fset$.
Since $a_{i,j}$ and $j$ are adjacent in $H$, we have $\langle u_{a_{i,j}}, u_j \rangle = 0$, and yet $\langle u_{a_{i,j}}, u_{a_{i,j}} \rangle \neq 0$. By considering the inner product of both sides of the above equality with $u_{a_{i,j}}$, it follows that $\alpha_3 = 0$, hence
\begin{eqnarray}\label{eq:u_j}
u_j = \alpha_1 \cdot u_i + \alpha_2 \cdot u_{b_{i,j}}.
\end{eqnarray}
Similarly, for some coefficients $\beta_1, \beta_2 \in \Fset$ we have
\begin{eqnarray}\label{eq:u_c}
u_{c_{i,j}} = \beta_1 \cdot u_i + \beta_2 \cdot u_{a_{i,j}},
\end{eqnarray}
and for some coefficients $\gamma_1, \gamma_2 \in \Fset$ we have
\begin{eqnarray}\label{eq:u_d}
u_{d_{i,j}} = \gamma_1 \cdot u_{a_{i,j}} + \gamma_2 \cdot u_{b_{i,j}}.
\end{eqnarray}

Now, since $j$ is adjacent in $H$ to both $a_{i,j}$ and $c_{i,j}$, it follows that $\langle u_j, u_{a_{i,j}} \rangle = \langle u_j, u_{c_{i,j}} \rangle = 0$. By~\eqref{eq:u_c}, it follows that $\langle u_j, \beta_1 \cdot u_i \rangle  = 0$. If $\langle u_i, u_j \rangle = 0$, then we are done. Otherwise, we must have $\beta_1 = 0$, hence $u_{c_{i,j}} = \beta_2 \cdot u_{a_{i,j}}$ where $\beta_2 \neq 0$.

Similarly, since $d_{i,j}$ is adjacent in $H$ to both $i$ and $j$, it follows that $\langle u_{d_{i,j}}, u_i \rangle = \langle u_{d_{i,j}}, u_j \rangle = 0$.
By~\eqref{eq:u_j}, it follows that $\langle u_{d_{i,j}}, \alpha_2 \cdot u_{b_{i,j}} \rangle  = 0$. If $\alpha_2 = 0$, then $u_j = \alpha_1 \cdot u_i$ and we are done. Otherwise, we must have $\langle u_{d_{i,j}}, u_{b_{i,j}} \rangle  = 0$.

We are left with the case where $u_{c_{i,j}} = \beta_2 \cdot u_{a_{i,j}}$ and $\langle u_{d_{i,j}}, u_{b_{i,j}} \rangle  = 0$.
Since $d_{i,j}$ is adjacent in $H$ to $c_{i,j}$, it follows that $\langle u_{d_{i,j}}, u_{c_{i,j}} \rangle  = 0$, and thus $\langle u_{d_{i,j}}, u_{a_{i,j}} \rangle  = 0$.
Using~\eqref{eq:u_d}, this contradicts the fact that $u_{d_{i,j}}$ is not self-orthogonal, completing the proof.
\end{proof}

We are ready to deduce the following.
\begin{lemma}\label{lemma:sound1}
For every field $\Fset$, if $\od_l(G',\Fset) \leq 3$ then $\chi(G) \leq 3$.
\end{lemma}

\begin{proof}
Let $\Fset$ be a field, and suppose that $\od_l(G',\Fset) \leq 3$. Then, there exists an orthogonal representation $(u_v)_{v \in V(G')}$ of $G'$ over $\Fset$ with locality at most $3$. Since the vertices $w,t,f$ form a triangle in $G'$, their vectors are pairwise orthogonal over $\Fset$.
Since they are not self-orthogonal, this implies that the subspace $U = \linspan(u_w, u_t, u_f)$ satisfies $\dim(U) = 3$.

We claim that all the vectors of the given orthogonal representation lie in $U$.
Before proving it, we explain why this yields that $\chi(G) \leq 3$.
Indeed, for every two vertices $i \in \{w,t,f\}$ and $j \notin \{w,t,f\}$ of $G$, the vectors of the corresponding gadget $H_{i,j}$ in $G'$ lie in the subspace $U$, hence by Lemma~\ref{lemma:Hij}, their vectors $u_i$ and $u_j$ satisfy either $\langle u_i,u_j \rangle = 0$ or $u_i = \alpha \cdot u_j$ for some $\alpha \in \Fset$.
Hence, for every vertex $v$ of $G$, the vector $u_v$ is a scalar multiple of precisely one of the pairwise orthogonal vectors $u_w$, $u_t$, $u_f$. This gives us a partition of the vertex set of $G$ into three color classes of a proper coloring of $G$, hence $\chi(G) \leq 3$.

It remains to show that all the vectors of the given orthogonal representation of $G'$ lie in $U$.
To this end, we say that a vertex of $G'$ {\em spans $U$} if the linear span of the vectors of its closed neighborhood contains $U$.
Observe that if a vertex spans $U$ then all the vectors of its neighbors lie in $U$, as otherwise we would get a contradiction to the locality $3$ of the orthogonal representation. The vertices $w,t,f$ clearly span $U$. This implies that for each $i \in [k]$, the vectors $u_{x_i}$ and $u_{\overline{x_i}}$ lie in $U$. Moreover, the literal vertices span $U$, because they belong to a triangle whose vertices are assigned vectors of $U$. Observe further that if the two base vertices of an OR gadget span $U$ then so does the top vertex. Indeed, in this case the top vertex belongs to a triangle whose vertices are adjacent to vertices that span $U$, hence each vertex of this triangle spans $U$. By the construction of $G$, repeatedly applying this property of the OR gadget implies that all the vertices of $G$ span $U$.
Finally, observe that if the vertices $i$ and $j$ of a gadget $H_{i,j}$ in $G'$ span $U$ then so do the rest of the vertices of the gadget.
Hence, all the vertices of $G'$ span $U$, and since every vertex of $G'$ has a neighbor, it follows that all the vectors of the orthogonal representation lie in $U$, as desired.
\end{proof}

To complete the soundness proof, we need the following lemma.
\begin{lemma}\label{lemma:sound2}
If $\chi(G) \leq 3$ then $\varphi$ is satisfiable.
\end{lemma}

\begin{proof}
Suppose that $\chi(G) \leq 3$, that is, there exists a proper coloring of $G$ with three colors, and assume, without loss of generality, that the vertices $w, t,f$ are colored by the colors $w, t, f$ respectively. By construction, every literal vertex is colored either $t$ or $f$, where its color is different from the color of its negation. This allows us to consider the assignment that assigns to every literal whose vertex is colored $t$ the value `true' and to every literal whose vertex is colored $f$ the value `false'.
We claim that this assignment satisfies $\varphi$.
To see this, suppose by contradiction that all the literals of some clause of $\varphi$ are assigned `false'.
Observe that the OR gadget satisfies that if its two base vertices are colored $f$ then so is the top vertex.
Repeatedly applying this observation, we get that the top vertex of the last OR gadget of this clause is colored $f$.
However, by the construction of $G$, this vertex is identified with the vertex $t$ whose color is $t$, contradicting our assumption.
\end{proof}

The following corollary summarizes the soundness of the reduction and follows immediately from Lemmas~\ref{lemma:sound1} and~\ref{lemma:sound2}.
\begin{corollary}\label{cor:sound}
For every field $\Fset$, if $\od_l(G',\Fset) \leq 3$ then $\varphi$ is satisfiable.
\end{corollary}

\subsubsection{Proof of Theorem~\ref{thm:hardness}}\label{sec:all_together}

\begin{proof}
Let $\Fset$ be a field.
For $k = 3$, the $\NP$-hardness result follows from the reduction presented in Section~\ref{sec:reduction}, whose correctness follows from Corollaries~\ref{cor:complete} and~\ref{cor:sound}.

For $k \geq 4$, consider the reduction from $\SAT$ that given a CNF formula $\varphi$ outputs the graph $G''$ obtained from the graph $G'$, defined in Section~\ref{sec:reduction}, by adding to it a complete graph on $k-3$ vertices and connecting them to all the vertices of $G'$. The graph $G''$ can obviously be constructed efficiently.
We turn to prove the correctness of the reduction.

For completeness, recall that if $\varphi$ is satisfiable, then by Lemmas~\ref{lemma:complete1} and~\ref{lemma:complete2}, $\chi(G') \leq 3$. A proper coloring of $G'$ can be extended to a proper coloring of $G''$ by assigning to the $k-3$ vertices of the added complete graph new distinct colors. It follows that $\chi(G'') \leq k$, so in particular, using~\eqref{eq:od_l_vs_chi_l}, $\od_l(G'',\Fset) \leq k$.

For soundness, suppose that $\od_l(G'',\Fset) \leq k$, that is, there exists an orthogonal representation of $G''$ over $\Fset$ with locality at most $k$.
Consider some vertex of $G''$ that was added to $G'$, and let $U$ be the subspace spanned by the vectors of its closed neighborhood.
By locality, we have $\dim(U) \leq k$. Let $W$ be the subspace of $U$ that consists of all the vectors of $U$ that are orthogonal to the $k-3$ vectors of the vertices of the added complete graph. Since these vectors are pairwise orthogonal, are not self-orthogonal, and belong to $U$, it follows that $\dim(W) \leq k-(k-3)=3$. In addition, since the vertices of $G'$ are adjacent to the vertices of the added complete graph, it follows that all of their vectors lie in $W$. Hence, the restriction of the given orthogonal representation to the vertices of $G'$ yields that $\od_l(G',\Fset) \leq \dim(W) \leq 3$, so by Corollary~\ref{cor:sound}, $\varphi$ is satisfiable, and we are done.
\end{proof}

\section{Local Orthogonality Dimension and Index Coding}\label{sec:index_coding}

We start this section with a brief background on the minrank parameter, whose study is motivated by the linear index coding problem.
We then review, in a generalized form, the connection given in~\cite{shanmugamKDALM2013} between the local chromatic number and the minrank parameter.
Finally, we relate the local orthogonality dimension of graphs to the minrank parameter and prove Theorem~\ref{thm:binaryOD}.

\subsection{Minrank}\label{sec:minrank}

It was proved in~\cite{BBJK06} that the optimal length of a solution to the linear index coding problem over a field $\Fset$ is precisely the minrank of the corresponding side information graph over $\Fset$. The minrank parameter is defined as follows.

\begin{definition}[Minrank~\cite{Haemers81}]\label{def:minrank}
Let $G=(V,E)$ be a graph on the vertex set $V=[n]$ and let $\Fset$ be a field.
We say that a matrix $M \in \Fset^{n \times n}$ {\em represents} $G$ if $M_{i,i} \neq 0$ for every $i \in V$, and $M_{i,j}=0$ for every distinct non-adjacent vertices $i$ and $j$.
The {\em minrank} of $G$ over $\Fset$ is defined as
\[{\minrank}_\Fset(G) =  \min\{{\rank}_{\Fset}(M)\mid M \mbox{ represents }G\mbox{ over }\Fset\}.\]
\end{definition}

We remark that the orthogonality dimension of a graph over a field $\Fset$ is closely related to the minrank over $\Fset$ of the complement graph.
To see this, consider a graph $G$ on $n$ vertices, and let us interpret the rows of a matrix $A \in \Fset^{n \times t}$ as an assignment of $t$-dimensional vectors to the vertices of $G$. It follows that $\od(G,\Fset)$ is the smallest integer $t$ for which there exists a matrix $A \in \Fset^{n \times t}$ such that $A \cdot A^T$ represents the complement $\overline{G}$ of $G$.
On the other hand, Definition~\ref{def:minrank} implies that ${\minrank}_\Fset(\overline{G})$ is the smallest integer $t$ for which there exist two matrices $A$ and $B$ in $\Fset^{n \times t}$ such that $A \cdot B^T$ represents $\overline{G}$ over $\Fset$.
With these definitions in mind, the only difference between $\od(G,\Fset)$ and ${\minrank}_\Fset(\overline{G})$ is in the requirement of $A$ and $B$ to be equal in the former. In particular, for every graph $G$ and every field $\Fset$, ${\minrank}_\Fset(\overline{G}) \leq \od(G,\Fset)$.

Another characterization of the minrank parameter is given by the following lemma and relies on the notion of independent representations of graphs (recall Definition~\ref{def:ind_rep}).
For a proof, see, e.g.,~\cite[Lemma~5]{AlishahiM21}.

\begin{lemma}\label{lemma:minrank}
For every graph $G$ and a field $\Fset$, ${\minrank}_\Fset(\overline{G})$ is the smallest integer $t$ for which there exists a $t$-dimensional independent representation of $G$ over $\Fset$.
\end{lemma}

\subsection{Local Chromatic Number and Minrank}

The following statement shows how a proper coloring of a graph can be combined with an appropriate collection of vectors to obtain an upper bound on the minrank of the complement graph. The proof relies on an idea of~\cite{shanmugamKDALM2013}.

\begin{proposition}\label{prop:local_minrank_gen}
Let $G=(V,E)$ be a graph and let $\Fset$ be a field.
Suppose that there exist a proper coloring $c: V \rightarrow [m]$ of $G$ and $m$ vectors $u_1,\ldots,u_m \in \Fset^t$, such that for every vertex $v \in V$ the vectors of
\[\{ u_{c(v')} \mid v' \in \{v\} \cup N(v)\},\]
that are indexed by the colors of the closed neighborhood of $v$, are linearly independent over $\Fset$. Then,
\[{\minrank}_\Fset(\overline{G}) \leq t.\]
\end{proposition}

\begin{proof}
Given a coloring $c: V \rightarrow [m]$ and vectors $u_1,\ldots,u_m \in \Fset^t$ as above, assign to each vertex $v \in V$ the vector $u_{c(v)}$. The assumption implies that for every vertex $v \in V$, the vector $u_{c(v)}$ of $v$ does not belong to the linear span of the vectors of its neighbors in $G$. It thus follows that this assignment forms a $t$-dimensional independent representation of $G$ over $\Fset$, hence by Lemma~\ref{lemma:minrank}, $\minrank_\Fset(\overline{G}) \leq t$, as desired.
\end{proof}

As an immediate application of Proposition~\ref{prop:local_minrank_gen}, we obtain the following.

\begin{proposition}\label{prop:local_minrank}
Let $G=(V,E)$ be a graph and let $\Fset$ be a field.
Suppose that there exist a proper coloring $c: V \rightarrow [m]$ of $G$ with locality $\ell$ and $m$ vectors in $\Fset^t$ such that every $\ell$ of them are linearly independent over $\Fset$. Then, $\minrank_\Fset(\overline{G}) \leq t$.
\end{proposition}

It is well known that for all integers $\ell \leq m$ and for every field $\Fset$ of size at least $m$, there exist $m$ vectors in $\Fset^\ell$ such that every $\ell$ of them are linearly independent over $\Fset$. Indeed, let $\alpha_1,\ldots,\alpha_m \in \Fset$ be distinct elements of the field, and consider the vectors $(1,\alpha_i,\alpha_i^2, \ldots, \alpha_i^{\ell-1}) \in \Fset^\ell$ for $i \in [m]$. By standard properties of Vandermonde matrices, every $\ell$ of these vectors are linearly independent over $\Fset$.
Applying Proposition~\ref{prop:local_minrank} with such collections of vectors, we derive the following.

\begin{proposition}\label{prop:local_minrank_n}
Let $G$ be a graph on $n$ vertices and let $\Fset$ be a field.
If there exists a proper coloring $c: V \rightarrow [m]$ of $G$ with locality $\ell$ such that $|\Fset| \geq m$, then $\minrank_\Fset(\overline{G}) \leq \ell$.
In particular, if $|\Fset| \geq n$, then $\minrank_\Fset(\overline{G}) \leq \chi_l(G)$.
\end{proposition}

As mentioned earlier, it was shown in~\cite{SimonyiT06} that there exist topologically $t$-chromatic graphs $G$ satisfying $\chi(G)=t$ and yet $\chi_l(G) = \lceil t/2 \rceil +1$.
By Proposition~\ref{prop:local_minrank_n}, it follows that they also satisfy ${\minrank}_\Fset(\overline{G}) \leq \lceil t/2 \rceil +1$ for every sufficiently large field $\Fset$. This shows that the lower bound on the minrank of the complements of topologically $t$-chromatic graphs, given in~\eqref{eq:AlishahiM} and proved in~\cite{AlishahiM21}, is tight on these graphs.

However, if the field $\Fset$ is not sufficiently large, collections of vectors in $\Fset^\ell$ such that every $\ell$ of them are linearly independent over $\Fset$ do not exist (see, e.g.,~\cite{Ball12}). Hence, given a graph $G=(V,E)$, a proper coloring $c: V \rightarrow [m]$ with locality $\ell$, and a field $\Fset$, one would like to apply Proposition~\ref{prop:local_minrank_gen} with a collection of $m$ vectors $u_1,\ldots,u_m \in \Fset^t$ of as small dimension $t$ as possible, such that the vectors associated with the colors of every closed neighborhood are linearly independent over $\Fset$. Interestingly, the problem of minimizing the dimension $t$ with respect to such constraints was studied by Schulman~\cite{Schulman92} (see also~\cite{HavivL19}) in the context of economical constructions of sample spaces of binary vectors with uniform restrictions to given sets of coordinates.
The following simple lemma is essentially given in~\cite{Schulman92}, and we include here a quick proof for completeness.

\begin{lemma}[\cite{Schulman92}]\label{lemma:Schulman}
Let $\calH$ be a collection of $h$ subsets of $[m]$, of size at most $\ell$ each, and let $\Fset$ be a finite field of size $q$. Put $t = \ell+ \lceil \log_q h \rceil$.
Then, there exist vectors $u_1, \ldots, u_m \in \Fset^t$ such that for every $H \in \calH$, the vectors of $\{ u_i \mid i \in H\}$ are linearly independent over $\Fset$.
\end{lemma}
\begin{proof}
The required $m$ vectors $u_1, \ldots, u_m \in \Fset^t$ are chosen one by one in $m$ steps as follows.
For every $j \in [m]$, the vector $u_j$ is chosen to be any vector of $\Fset^t$ such that for every $H \in \calH$ with $j \in H$, $u_j$ does not lie in the linear span of the vectors of $\{ u_i \mid i \in H,~i<j\}$. Note that the vectors of such a set span a subspace of dimension at most $\ell-1$. Hence, the total number of forbidden vectors in every step is at most $h \cdot q^{\ell-1} < q^t$, where the inequality follows from the definition of $t$. This ensures that the $m$ vectors can be chosen successfully, satisfying the required condition.
\end{proof}

Now, by combining Proposition~\ref{prop:local_minrank_gen} with Lemma~\ref{lemma:Schulman}, we reprove the result of~\cite{shanmugamKDALM2013} for the binary field (with a slightly better multiplicative constant).
\begin{corollary}[\cite{shanmugamKDALM2013}]\label{cor:localSDL}
For every graph $G$ on $n$ vertices, $\minrank_{\Fset_2}(G) \leq \chi_l(\overline{G}) + \lceil \log_2 n \rceil$.
\end{corollary}
\begin{proof}
Let $G=(V,E)$ be a graph $n$ vertices, and put $\ell = \chi_l(G)$. Then, for some integer $m$ there exists a proper coloring $c:V \rightarrow [m]$ of $G$ with locality $\ell$.
Let $\calH$ be the collection of the $n$ sets of colors of the closed neighborhoods of the vertices of $G$. By locality, the size of every such set is at most $\ell$, hence by Lemma~\ref{lemma:Schulman}, there exist vectors $u_1, \ldots, u_m \in \Fset_2^t$ for $t = \ell+ \lceil \log_2 n \rceil$, such that the vectors of $\{ u_i \mid i \in H\}$ are linearly independent over $\Fset_2$ for every $H \in \calH$. Applying Proposition~\ref{prop:local_minrank_gen}, it follows that $\minrank_{\Fset_2}(\overline{G}) \leq t$. The proof is completed by switching the roles of $G$ and of its complement.
\end{proof}

\subsection{Local Orthogonality Dimension and Minrank}

We finally prove Theorem~\ref{thm:binaryOD}, showing that the local chromatic number in Corollary~\ref{cor:localSDL} can be replaced by the local orthogonality dimension over $\Fset_2$.
The proof uses a probabilistic argument. We first prove the following lemma.

\begin{lemma}\label{lemma:independent}
For a finite field $\Fset$ and integers $t$ and $m$, let $D \subseteq \Fset^t$ be a set of linearly independent vectors over $\Fset$, and let $A \in \Fset^{m \times t}$ be a uniformly chosen random matrix over $\Fset$.
Then, the vectors $A \cdot w$ with $w \in D$ are distributed uniformly and independently over $\Fset^m$.
\end{lemma}

\begin{proof}
Let $D = \{w^{(1)},\ldots,w^{(r)}\} \subseteq \Fset^t$ be a set of $r$ linearly independent vectors over $\Fset$, and let $A \in \Fset^{m \times t}$ be a uniformly chosen random matrix over $\Fset$.
It is clear that for every $i \in [r]$, the vector $w^{(i)}$ is nonzero, hence the vector $A \cdot w^{(i)}$ is uniformly distributed over $\Fset^m$.
To prove that the vectors $A \cdot w^{(i)}$ with $i \in [r]$ are distributed independently, it suffices to show that the number of matrices $A \in \Fset^{m \times t}$ satisfying $A \cdot w^{(i)} = b^{(i)}$ for all $i \in [r]$ for an $r$-tuple of vectors $b^{(1)},\ldots,b^{(r)} \in \Fset^m$ is independent of the choice of the $r$-tuple.

To do so, fix such $b^{(1)},\ldots,b^{(r)} \in \Fset^m$ and let $a^{(1)},\ldots,a^{(m)}$ denote the $m$ random rows of the matrix $A$. Notice that for every $i \in [r]$, it holds that $A \cdot w^{(i)} = b^{(i)}$ if and only if
\begin{eqnarray}\label{eq:a^j}
\langle a^{(j)} , w^{(i)} \rangle = b^{(i)}_j
\end{eqnarray}
for all $j \in [m]$. Since the vectors of $D$ are linearly independent over $\Fset$, it follows that for every $j \in [m]$, the vectors $a^{(j)}$ satisfying~\eqref{eq:a^j} for all $i \in [r]$ form an affine subspace over $\Fset$ of dimension $t-r$, hence their number is precisely $|\Fset|^{t-r}$. This implies that the number of matrices $A$ satisfying $A \cdot w^{(i)} = b^{(i)}$ for all $i \in [r]$ is $|\Fset|^{m \cdot (t-r)}$, independently of the vectors $b^{(1)},\ldots,b^{(r)}$, as required.
\end{proof}

We are ready to prove the following theorem, which confirms Theorem~\ref{thm:binaryOD}.
\begin{theorem}\label{thm:minrk_localOD_gen}
For every graph $G$ on $n$ vertices and a finite field $\Fset$ of size $q$,
\[{\minrank}_{\Fset}(G) \leq \od_l(\overline{G},\Fset)+\lceil \log_q n \rceil.\]
\end{theorem}

\begin{proof}
For a graph $G=(V,E)$ on $n$ vertices and a finite field $\Fset$ of size $q$, put $\ell = \od_l(G,\Fset)$ and $m = \ell + \lceil \log_q n \rceil$.
Then, for some integer $t$, there exists a $t$-dimensional orthogonal representation $(w_v)_{v \in V}$ of $G$ with locality $\ell$.
Let $A \in \Fset^{m \times t}$ be a uniformly chosen random matrix over $\Fset$.
For every vertex $v \in V$, define $\widetilde{w}_v = A \cdot w_v \in \Fset^m$.
We turn to show that, with positive probability, the vectors $(\widetilde{w}_v)_{v \in V}$ form an $m$-dimensional independent representation of $G$.

To this end, for a vertex $v \in V$, let $B(v) \subseteq N(v)$ be a collection of vertices whose vectors $\{w_u \mid u \in B(v)\}$ in the given orthogonal representation form a basis of $\linspan (\{w_u \mid u \in N(v)\})$. Since the locality of the given orthogonal representation is $\ell$, it follows that the dimension of this subspace is at most $\ell-1$, and thus $|B(v)| \leq \ell-1$.
Now, for each vertex $v \in V$, let $L_v$ denote the event that
\[\widetilde{w}_v \in \linspan (\{\widetilde{w}_u \mid u \in N(v)\}).\]
By the definition of the set $B(v)$, it follows that
\begin{eqnarray}\label{eq:L_v}
\linspan (\{\widetilde{w}_u \mid u \in N(v)\}) = \linspan (\{\widetilde{w}_u \mid u \in B(v)\}).\end{eqnarray}
Moreover, by $|B(v)| \leq \ell-1$, the dimension of this subspace is at most $\ell-1$, hence its size is at most $q^{\ell-1}$.
Since the vectors $w_u$ with $u \in \{v\} \cup B(v)$ are linearly independent over $\Fset$, Lemma~\ref{lemma:independent} implies that the corresponding vectors $\widetilde{w}_u = A \cdot w_u$ are distributed uniformly and independently over $\Fset^m$. Hence, for every choice of the vectors $\widetilde{w}_u$ with $u \in B(v)$, the probability that the vector $\widetilde{w}_v$, which is distributed uniformly and independently of them, lies in the subspace they span is at most $q^{\ell-1}/q^m$. This yields, using~\eqref{eq:L_v}, that
\[\Prob{}{L_v} \leq \frac{q^{\ell-1}}{q^m} = \frac{q^{\ell-1}}{q^{\ell + \lceil \log_q n \rceil}} \leq \frac{1}{qn}.\]
By the union bound, the probability that there exists a vertex $v \in V$ for which the event $L_v$ occurs does not exceed $n \cdot \frac{1}{qn} = \frac{1}{q} < 1$.
This implies that there exists an $m$-dimensional independent representation of $G$ over $\Fset$, so by Lemma~\ref{lemma:minrank}, $\minrank_{\Fset}(\overline{G}) \leq m$.
The proof is completed by switching the roles of $G$ and of its complement.
\end{proof}

\section*{Acknowledgments}
We are grateful to Adam Chapman, Shoni Gilboa, and Fr{\'{e}}d{\'{e}}ric Meunier for helpful discussions and suggestions, and to the anonymous reviewers for their detailed and valuable comments on an earlier version of the paper.

\bibliographystyle{abbrv}
\bibliography{local}

\end{document}